\documentclass{amsproc}

\input{macros.tex}

\begin{document}

\setcounter{page}{39}

\title[Geodesics and Busemann functions in FPP]{Infinite geodesics, asymptotic directions, and\\ Busemann functions in first-passage percolation}


%    author one information
\author{Jack Hanson}
%\address{Department of Mathematics, City College of NY, USA}
\address{Department of Mathematics, City College of NY, USA\afterpage{\blankpage}}
%\curraddr{}
\email{jhanson@ccny.cuny.edu}
%\thanks{}

%    author two information
%\author{}
%\address{}
%\curraddr{}
%\email{}
%\thanks{}

%\subjclass[2000]{Primary }
%    The 2010 edition of the Mathematics Subject Classification is
%    now available.  If you are citing a classification from the
%    new scheme, use the following input coding instead.

\date{\today}

\begin{abstract}
We show existence, uniqueness, and directedness properties for infinite geodesics in the FPP model. After giving the fundamental definitions, we describe results by Newman and collaborators giving existence and uniqueness of directed geodesics under an unproven curvature assumption. We then give two proofs of the existence of at least two infinite geodesics under no unproven assumptions. In the final two sections, we give proofs of directedness statements for infinite geodesics using more recent methods which give information even under no unproven assumptions and prove a generalized uniqueness statement for infinite geodesics.
\end{abstract}

\maketitle


%\setcounter{tocdepth}{2}
\tableofcontents

\section{Motivation}
We will concern ourselves with the usual FPP model. In this chapter, we consider the following broad question: how do long geodesics behave? Recall that (see \cite[Section 4]{ch:Damron}) it is expected that the geodesic from $0$ to $nx$ should deviate from the straight line parallel to $x$ by a distance of order $n^{\xi}$. One should reasonably expect that this exponent gives uniform control of the geodesic close to its endpoints, and that $n^{\xi}$ is an over-estimate of wandering far from the midpoint of the geodesic. For instance, if $m \ll n$, the first $m$ vertices on the geodesic should lie within distance $m^{\xi}$ of the line parallel to $x$. This would represent a form of tightness for the family of geodesics from $0$ to $nx$ as $n \to \infty$, leading naturally to the question of whether these geodesics 
\index{geodesic!convergence}%
converge. That is, whether for fixed $K$, the first $K$ steps of the geodesics to each $nx$ are identical for all large $n$. The resulting limiting path is referred to as an
\index{geodesic!infinite}%
infinite geodesic (some precise definitions are given below in Section \ref{sec:maindef}, and others are recalled from \cite{ch:Damron}).

More generally, one can ask about what classes of limits are possible for different sequences of finite geodesics. A path from $0$ to $n e_1$ which ``backtracks'' too far in the $-e_1$ direction seems unlikely to be a geodesic, since it should be likely to accumulate excessive passage time the further it backtracks. In particular, the family of geodesics from $0$ to $n e_1$ and the family from $0$ to $- ne_1$ should admit subsequential limits which are distinct in a strong sense: sharing only finitely many initial edges or vertices. Analogous conjectures for roughly parallel sequences of points seem trickier to make. Should the geodesic from $0$ to $n e_1 + \sqrt{n} e_2$ produce yet another distinct limit with positive probability? We shall see that, under strong assumptions on the model's limit shape, one expects a unique limiting geodesic corresponding to each direction in the plane, so the preceding question should have a negative answer in two dimensions. 

A further motivation for these questions comes from their relationship to other important questions in the model. As we will discuss, properties of infinite geodesics are closely related to those of the FPP limiting shape. Because of the relationship between geodesic wandering and fluctuations of the passage time (recall the scaling relation
\index{scaling relation}%
 $\chi = 2 \xi - 1$), one also could hope that techniques for controlling geodesic wandering will be useful for estimating fluctuations. Another issue already discussed (see \cite[Section 1.2]{ch:Damron}) is the low-temperature behavior of the so-called 
 \index{disordered ferromagnet}%
 ``disordered ferromagnet'' model from statistical physics. According to scaling arguments, the two-dimensional disordered ferromagnet should have a unique pair of ground states, which would imply that 2d FPP admits no doubly infinite geodesic. While doubly infinite geodesics will not be a focus of this chapter, the tools we develop have direct application to this question, since a doubly infinite geodesic should resemble a singly infinite geodesic near its midpoint.

\index{Busemann function}%
Busemann functions (and certain ``Busemann-like'' objects) will be a central tool in our study. These functions were originally introduced in the context of nonrandom metric geometry, and were brought to the study of FPP by Hoffman \cite{H05}. A Busemann function has the form $B(x,y) = \lim_{n \to \infty}[T(x, x_n) - T(y, x_n)]$, where $(x_n)$ is the sequence of vertices in some infinite geodesic, ordered as they appear on the geodesic. $B(x,y)$ can be interpreted as the relative distance to infinity of $x$ and $y$ along this geodesic. By bounding the growth of $B$, one can control where limiting geodesics are allowed to wander. Limits of the form $\lim_{n \to \infty} [T(x, n e_1) - T(y, n e_1)]$ (note the sequence $(n e_1)$ is typically not an infinite geodesic) were studied earlier by Newman \cite{N97}, since such a limit gives information about the ``microstructure'' of the surface of the growing region in the FPP growth model.

\section{Geodesics and infinite geodesics}
\subsection{Setting and assumptions}\label{sec:maindef}
We consider the setting of FPP on $(\Z^d, \cE^d)$ with $d \geq 2$ --- we recall some of the notation here for readability. For each $e \in \cE^d$, there is an associated 
\index{edge weights}%
edge weight $t_e \geq 0$. We assume for definiteness that the collection $(t_e)$ is independent and identically distributed with common distribution $F$. However, it is important to note that in many of the results below, the i.i.d. restriction can be considerably relaxed to more general translation-invariant cases. We generally write $\prob$ for the joint distribution of the edge weights.

In all cases that follow, we consider $F$ having at least the following properties.
\begin{assumption}[Standing assumptions]\label{asn:std} Unless otherwise specified, from now on we assume:
%\hspace{1em}
%\smallskip
\  \begin{enumerate}
  \item $\bfE \min\{t_1, \, \ldots,\, t_{2d}\}^d < \infty$. 
    \item $F$ is continuous --- i.e., $\prob(t_e = a) = 0$ for each real $a$.
  \end{enumerate}
\end{assumption}
Note that (1) from Assumption~\ref{asn:std} is the moment assumption made in the statement of the shape theorem (see \cite[Section 3]{ch:Damron}). Item (2) is stronger than the assumption $\prob(t_e = 0) < p_c$ from that theorem, so these assumptions give us the existence of a 
\index{shape theorem}%
\index{limit shape}%
limit shape $\cB$ and a norm $g$ such that $\cB = \{x: \, g(x) \leq 1\}.$

For some results (the results of Newman below --- see Theorem \ref{thm:newman}), we will need the following strong assumptions in addition to the above.
\begin{assumption}[Strong assumptions]\label{asn:strong}
%\hspace{1em}
%\smallskip 
\ \begin{enumerate}%[leftmargin=0pt]
\setlength{\itemindent}{25pt}
\item[{\rm(}ExpM{\rm)}] There is some $\alpha > 0$ such that $\bfE \exp(\alpha t_e) < \infty$.
\item[{\rm(}Curve{\rm)}] The 
\index{limit shape}%
limit shape $\cB$ is uniformly curved, in the following sense. There exists a uniform constant $c = c_{\cB}$ such that, for any $z_1$ and $z_2$ having $g(z_1) = g(z_2) = 1$ and any $a \in [0,1]$,
\[g(z) \leq 1 - c\left[\min\{g(z - z_1),\, g(z - z_2)\} \right]^2\ , \]
where $z = a z_1 + (1-a)z_2$.
\end{enumerate}
\end{assumption}

Some comments on these assumptions are in order. Assumption (ExpM) is straightforward but strong, and it is (as we will see) unnecessary for many of the results on infinite geodesics. In fact, it is mainly applied to guarantee concentration results for passage times hold, and any result which relies too heavily on these will not adapt well to the non-i.i.d. settings mentioned above. 

Assumption~(Curve) is widely believed to hold for most ``reasonable'' edge weight distributions, and this claim is bolstered by comparison to some exactly solvable models closely related to FPP. But as it stands, (Curve) is far from proof. Indeed, as previously discussed, it is not even currently known that there exists some $F$ such that $\cB$ is strictly convex for i.i.d. edge weights, though this is also highly plausible.

\subsection{Types of geodesics}

We recall the definition of a geodesic between two terminal points $x$ and $y$. In what follows, we identify a path in $\Z^d$ with either the sequence of vertices or edges it traverses, depending on which is more convenient in context.
\begin{definition}
  A nearest-neighbor path $\gamma = (x = x_0, \, x_1,\, \ldots,\, x_n = y)$ is called a 
  \index{geodesic}%
  \index{geodesic!finite}%
  {\bf (finite) geodesic} between $x$~and~$y$ if $T(\gamma) = T(x,y)$.
\end{definition}

\begin{lemma}\label{lem:unique}
  Under Assumption \ref{asn:std}, there a.s. exists a unique finite geodesic between each $x$ and $y$ in $\Z^d$. Moreover, for any two finite paths $\gamma_1 \neq \gamma_2$, we a.s. have $T(\gamma_1) \neq T(\gamma_2)\ .$
\end{lemma}
We call the event on that all finite paths have distinct passage times and that there exists a unique geodesic between every pair of points of $\Z^d$  by the name $\Omega_u$; in this language, Lemma \ref{lem:unique} tells us that $\prob(\Omega_u) = 1$. In many of the arguments that follow, we work conditionally on the event $\Omega_u$ without further comment.

\begin{proof}
  The existence claim follows easily from the shape theorem; we sketch the argument. Namely, we have under our assumptions that (a.s.) $T(x, z) \rightarrow \infty$ as $|z| \rightarrow \infty$. On the other hand, fixing a nonrandom shortest lattice path (in the $\Z^d$ metric, not the FPP metric) $\gamma$ from $x$ to $y$, we have that $T(x,y) \leq T(\gamma)$. In particular, for a.e. realization of the edge weights, there is some Euclidean ball such that $T(x,z) > T(x,y)$ for any $z$ outside this ball. This implies that the infimum over paths in the definition of $T(x,y)$ is a.s. effectively over a finite set and ensures the existence of the geodesic.

We show uniqueness. Suppose that $\gamma_1 \neq \gamma_2$ are two finite nearest-neighbor paths from $x$ to $y$ --- in particular, that some edge $e \in \gamma_1 \setminus \gamma_2$. For both $\gamma_1$ and $\gamma_2$ to be geodesics from $x$ to $y$, we must have $T(\gamma_1) = T(\gamma_2)$; we show that this has probability zero of occurring.

Condition on the values of the weights of all edges other than $e$, and view $h = T(\gamma_1) - T(\gamma_2)$ conditionally as a function of $t_e$.  We have $h'(t_e) = 1$ for all values of $t_e$, so there exists at most one value $a$ such that $h(a) = 0$. But $\prob(t_e = a) = 0 $, so there is conditional probability $0$ that $T(\gamma_1) = T(\gamma_2)$.
\end{proof}

\begin{definition}
  The unique geodesic from $x$ to $y$ whose existence is guaranteed by Lemma \ref{lem:unique} is denoted $G(x,y)$.
\end{definition}

Of course, a subpath of a geodesic is also a geodesic. In particular, if $x_i$ and $x_j$ are vertices of $G(x,y)$, then the subpath of $G(x,y)$ between them is identical to $G(x_i, x_j)$.

One can ask many interesting questions about the behavior of $G(x,y)$ --- for instance, how far does it typically deviate or ``wander'' from the straight line segment between $x$ and $y$? One such question of particular importance, which is closely related to the preceding wandering question, is:
\begin{itemize}
\item For what sequences $(y_n)_n \subseteq \Z^d$ with $|y_n| \rightarrow \infty$ does the sequence of geodesics $(G(x,y_n))_n$ 
\index{geodesic!convergence}%
converge?
\end{itemize}

Here we say a sequence of (nearest-neighbor) paths $(\gamma_n)_n$ converges to a path $\gamma$ if for each fixed $k$ and all large $n$, the first $k$ edges of $\gamma_n$ are identical with the first $k$ edges of $\gamma$. It is not hard to see (and we make precise in the proof of Proposition \ref{prop:onegeo} below) that the following holds for almost every edge weight configuration: if there are a vertex $x$, a sequence $(y_n)_n$, and some infinite path $\gamma$ such that $\lim_n G(x,y_n) = \gamma$, then $\gamma$ must be an infinite geodesic, in the sense of the next definition.

\begin{definition}
  An infinite path $\gamma = (x_0,\, x_1,\, \ldots)$ is an 
  \index{geodesic!infinite}%
  \index{unigeodesic}%
  {\bf infinite geodesic} or {\bf unigeodesic} if its finite subsegments are all finite geodesics --- that is, for every $i < j$,
\[(x_i, \, x_{i+1},\ , \ldots,\, x_j) = G(x_i, x_j)\ . \]
\end{definition}
We will sometimes omit the words ``finite'' or ``infinite,'' simply saying ``geodesic'' to mean a finite or infinite geodesic, when the meaning is clear by context.

Of course, simpler than convergence statements would just be existence statements, and this is mainly what we will focus on. Do there exist unigeodesics? How many? And what can we say about their wandering?

\begin{proposition}\label{prop:onegeo}
  With probability one, there exists an infinite geodesic $\gamma = (0 = x_0, \, x_1, \ldots)$.
\end{proposition}
\begin{proof}
This is a standard 
\index{diagonal argument}%
``diagonal argument''.  Start with the sequence $(G(0, n e_1))_n. $ The second vertex of each $G(0, ne_1)$ is an element of $\{\pm e_i\}_{i = 1}^d$, a finite set; there thus exists an $e \in \{\pm e_i\}$ and a sequence $n_k \rightarrow \infty$ such that the first two vertices of $G(0, n_k)$ are identical for all $k$ --- i.e., $G(0, n_k e_1) = (0, e, \, \ldots)$ for all $k$. Repeating this argument on the subsequence $(G(0,n_k e_1))$ and inducting produces a subsequential limit $\gamma$. 

Since any finite subsegment of $\gamma$ is also a finite subsegment of $G(0, ne_1)$ for some $n$, and since finite subsegments of geodesics are also geodesics, we see that $\gamma$ is a unigeodesic.
\end{proof}

The proof of Proposition \ref{prop:onegeo} uses practically nothing from the model, but for this reason it is difficult to improve. One obvious consequence of translation-invariance is that we can replace $0$ by any $x$ and get an infinite geodesic from $x$ as well. But perhaps every unigeodesic is essentially the same. 

To clarify, let us call unigeodesics $\gamma_1$ and $\gamma_2$ 
\index{distinct}%
{\bf distinct} if $|\gamma_1 \triangle \gamma_2| = \infty$. It is easy to see that uniqueness implies that unigeodesics cannot touch, subsequently separate, then touch again, as this would imply that the separated segments are two different geodesics between vertices where they touch. This can be formalized as:
\begin{proposition}\label{prop:distinct}
  On the probability one event $\Omega_u$ of Lemma \ref{lem:unique},  geodesics $\gamma_1,\,\gamma_2$ are distinct if and only if they touch at most finitely often: $|\gamma_1 \cap \gamma_2| < \infty$. More specifically, if $\gamma_1 = (z_1, z_2, \ldots)$ and $\gamma_2 = (z_1', z_2', \ldots)$ are distinct, the intersection is a single subpath: for instance, $\gamma_1 \cap \gamma_2 = (z_K, \ldots, z_{K+\ell})$.
\end{proposition}
In particular, geodesics $\gamma_1$ and $\gamma_2$ (almost surely) cannot touch infinitely often without 
\index{geodesic!merging}%
\index{geodesic!coalescence}%
{\bf merging} or {\bf coalescing}: becoming asymptotically identical. We leave the formal proof of Proposition \ref{prop:distinct} as an exercise, since it gives some intuition about the ``treelike'' behavior of geodesics when they are known to be unique.
%Note that under our standing assumptions, distinct geodesics must eventually separate (i.e., $|\gamma_2 \cap \gamma_2| < \infty$). If $x,\, y \in \gamma_1 \cap \gamma_2$ then by uniqueness of finite geodesics, the subsegment of both $\gamma_1$ and $\gamma_2$ from $x$ to $y$ must be $G(x,y)$ --- this implies that unigeodesics cannot touch infinitely often without merging.

It is not clear at all a priori how to show that  there exist even two distinct unigeodesics. It is intuitively plausible, however, that subsequential limits of (for instance) $(G(0, n e_1))_{n > 0}$ and $(G(0, -ne_1))_{n > 0}$ should look very different. If some $\gamma$ arose as a subsequential limit of both sequences, it would have to in some sense represent a ``fast path'' in both the $e_1$ and $-e_1$ directions, which is difficult to imagine.

Let $\cN$ denote the random number of distinct unigeodesics (the supremal cardinality among collections of distinct unigeodesics). The above reasoning suggests that $\prob(\cN \geq 2) > 0$. It also suggests that limits of $G(0, ne_1)$ and $G(0, -ne_1)$ could be differently directed, in the following sense:
\begin{definition}
  Let $\gamma = (x_0,\, x_1,\, \ldots)$ be a unigeodesic. We say that $\theta$ is a 
  \index{limiting direction}%
  limiting direction
   of $\gamma$ if $\theta$ is a limit point of the sequence $(x_i / |x_i|)_i$. The set of all limiting directions for $\gamma$ shall be denoted $\Theta(\gamma)$.
\end{definition}

We will now focus on the following questions:
\begin{quest}\label{quest:geono}
  Can we make the above reasoning precise to show $\prob(\cN \geq 2) > 0$? What is the largest value $\cN$ can take --- for instance, is $\prob(\cN = \infty) > 0$?
\end{quest}
\begin{quest}\label{quest:geodir}
What $\Theta(\gamma)$ are realizable for unigeodesics $\gamma?$ Can we find $\gamma$ such that $\Theta(\gamma)$ is strictly smaller than $\bbS^{d-1}$? Or indeed, is there positive probability that there exists $\gamma$ with $\Theta(\gamma) = \bbS^{d-1}$? 
\end{quest}

\section{Geodesic trees and the curvature condition}
The first results on the above existence questions are due to Newman \cite{N97} (and related results with collaborators C. Licea \cite{LN96} and M. Piza \cite{NP95}). We say that a path $\gamma$ has 
\index{direction}%
{\bf direction} $\theta \in \bbS^{d-1}$ if $\Theta(\gamma) = \{\theta\}$. That is, if $\gamma = (x_0, \, x_1,\, \ldots)$, then $\gamma$ has direction $\theta$ if and only if $x_i / |x_i| \rightarrow \theta$.

\begin{theorem}[\cite{N97}]\label{thm:newman}
  Assume (ExpM) and (Curve) in addition to the usual assumptions. Then:
  \begin{enumerate}
  \item[{\rm(1)}] With probability one, for each $\theta \in \bbS^{d-1}$ there is a unigeodesic beginning at $0$ having direction $\theta$.
    \item[{\rm(2)}] With probability one, each unigeodesic has direction; i.e., for each geodesic $\gamma$, there exists a $\theta \in \bbS^{d-1}$ such that $\Theta(\gamma) = \{\theta\}$.
  \end{enumerate}
\end{theorem}

In addition to the above existence questions, one is often interested in the relationship of various geodesics from $0$ to one another: how large is the intersection between two infinite or long finite geodesics. In fact, we can simultaneously consider the collection of all geodesics $\{G(0,x)\}$ as a graph in its own right. This is the perspective taken by Newman \cite{N97}, which will be useful in the study of properties of infinite geodesics. 

For any $x \in \Z^d$, define the 
\index{geodesic!tree}%
{\bf geodesic tree} $\cT(x)$ as the directed graph given by the union $\bigcup_{y \in \Z^2} G(x,y)$, where each $G(x,y)$ is considered as a path directed away from $x$. In other words, $\cT(x)$ is the graph whose vertex set is $\Z^2$ and whose (directed) edge set is the set of all ordered pairs $(z_i, z_{i+1})$ which appear in $G(x,y) = (z_0, z_1, \ldots, z_n)$ for some $y$. 
 As the name anticipates, this graph is a tree:
 \begin{lemma}\label{lem:geotree}
  Almost surely, each finite directed path in $\cT(x)$ is a geodesic. Moreover, considering $\cT(x)$ as an undirected graph, it is a tree.
 \end{lemma}
 \begin{proof}
   We consider an outcome in the existence and uniqueness event $\Omega_u$ from Lemma \ref{lem:unique}. We first prove that each directed path in $\cT(x)$ is a geodesic. We inductively prove the following stronger statement: each directed path $(y_0, y_1, \ldots, y_n)$ in $\cT(x)$ is the terminal segment of $G(x, y_n)$. For the base case $n=1$, we consider a single directed edge $(y_0, y_1)$ in $\cT(x)$. By definition, $(y_0, y_1)$ appears as a directed edge in some geodesic $G(x,y)$. Since subpaths of geodesics are geodesics, the segment of $G(x,y)$ up to the appearance of $y_1$ is the geodesic $G(x,y_1)$, proving the statement for $n  =1$.

Assume the statement holds for a positive integer $n$, and consider a directed path $(y_0, \ldots, y_{n+1})$ in $\cT(x)$. By the inductive hypothesis, $(y_1, \ldots, y_{n+1})$ is itself the terminal segment of $G(x, y_{n+1})$. The segment of $G(x, y_{n+1})$ up to $y_1$ must be $G(x, y_1)$ by uniqueness of geodesics. But since $(y_0, y_1)$ is a directed edge of $\cT(x)$, it must be the final directed edge of $G(x, y_1)$ by the induction hypothesis. Thus, $G(x, y_{n+1})$ is the concatenation of $G(x, y_0)$ with the directed path $(y_0, \ldots, y_{n+1})$, completing the inductive step.

We now prove the claim that (as an undirected graph) $\cT(x)$ is a tree. Clearly it is connected, since there is a directed path from $0$ to each $y$. Suppose for the sake of contradiction that there were some simple cycle $(z_1, z_2, \ldots, z_n = z_1)$ in $\cT(x)$. If this were also a directed cycle in either orientation, it would be $G(z_1, z_1)$ --- a contradiction to the uniqueness of geodesics. So we may assume that the orientation of the edges reverses as we traverse this cycle. 

Without loss of generality, we assume $(z_1, z_2)$ is a directed edge, and let $1 < K $ be the smallest integer such that $(z_K, z_{K+1})$ is not a directed edge. Then $(z_{K+1}, z_K)$ is a directed edge, and so is $(z_{K-1},z_K)$. In particular, both of these must be the last edge of $G(x, z_K)$. This clearly leads to a contradiction, since $(z_1, \ldots, z_n)$ is a simple cycle.
 \end{proof}

Since geodesics in distinct directions must be distinct, Theorem \ref{thm:newman} suggests that ``reasonable'' edge weight distributions have $\prob(\cN = \infty) = 1$. In other words, the tree $\cT(0)$ a.s. has infinitely many ends. The proof of the theorem amounts to showing that $\cT(0)$ is ``straight'', in the sense that finite paths in $\cT(0)$ stay close to to the Euclidean line segment between their endpoints. The result then follows using the limiting procedure described after Proposition \ref{prop:onegeo}, by taking limits of finite geodesics $G(0, v_n)$ with $v_n / |v_n| \rightarrow \theta$. These (or a subsequence thereof) converge to an infinite directed path in $\cT(0)$ , which is an infinite geodesic by Lemma \ref{lem:geotree}.

Implementing the strategy amounts to rigorously establishing a strong version of our heuristic that geodesics to points far off in direction $\theta$ should not be geodesics for points in direction $-\theta$ (``no backtracking''). However, one must show a version of this statement not just for $- \theta$, but also for directions even very close to $\theta$.

We will not attempt to prove the theorem here, but will give a very basic idea of the connections between regularity properties of $\cB$ and statements of a ``no backtracking'' flavor. Suppose we wish to show a much weaker statement; namely, that the geodesic from $0$ to $n e_1 + n e_2$ is not likely to be identical to the concatenation of the geodesic $G(0, n e_1)$ with $G(n e_1, n e_1 + n e_2)$. One way to do this would be to show that
\begin{equation}
\label{eq:basicsub}
T(0, n e_1 + n e_2) \ll T(0, n e_1) + T(n e_1, n e_1 + n e_2). 
\end{equation}

On can approximate the left- and right-hand sides of \eqref{eq:basicsub}, for $n$ large, by $n g(e_1 + e_2)$ and $n[g(e_1) + g(e_2)]$ (up to $o(n)$ terms). Then \eqref{eq:basicsub} will certainly hold if $g(e_1 + e_2) < g(e_1) + g(e_2)$ or equivalently if $g(e_1/2 + e_2/2) < (1/2)[g(e_1) + g(e_2)]$. This follows, for instance, if $\cB$ is strictly convex.

\subsection{Uniqueness: Licea-Newman}
\label{sec:lnuniq}
In addition to the above existence statements, one can ask for upper bounds on the number of infinite geodesics.  A natural question is whether many geodesics exist in a given (fixed, deterministic) direction $\theta$. At a heuristic level, this possibility seems unlikely on $\Z^2$. Since such geodesics exhibit power-law wandering (see Section 4.2 in \cite{ch:Damron}) %HERE
in the direction perpendicular to $\theta$, we would expect them to typically cross each other (unless their wandering were somehow synchronized). On the other hand, multiple such crossings would lead to a ``loop'' of two disjoint geodesics between the same endpoints, a contradiction under our weight assumptions. With this two-dimensional picture in mind, we restrict for the remainder of Section \ref{sec:lnuniq} to $\Z^2$.

 The work of Licea and Newman \cite{LN96} shows a limited form of a ``uniqueness'' statement for the geodesics in a fixed direction. Namely, when $d = 2$, we have, fixing $\theta$ in some full-measure subset of $\bbS^1$, that there is a.s. at most one distinct geodesic $\gamma$ such that $\Theta(\gamma) = \{\theta\}$. This will be a partial motivation for our assumption (LimG) introduced below as Assumption \ref{asn:limg}, which allows the construction of a well-behaved ``Busemann function''.

\begin{theorem}[Licea-Newman \cite{LN96}]\label{thm:lnunique}
Consider FPP on $\Z^2$ under the standard assumptions of Assumption \ref{asn:std}. There exists a deterministic set $D \subseteq [0, 2\pi)$ of Lebesgue measure $2 \pi$ such that the following holds. For any $\theta \in D$,  
\[\prob\left(\text{there are distinct $\theta$-directed unigeodesics } \gamma_1,\, \gamma_2\right) = 0. \]
\end{theorem}
Note that the $\theta$  appearing in the argument of the probability is deterministic. This is an important distinction; in fact the  analogue of Theorem \ref{thm:lnunique} with uniqueness simultaneously holding for all $\theta \in [0, 2\pi)$ should not hold. This is because two parallel geodesics should run along directions where there is a 
\index{competition interface}%
``competition interface'' (see Theorem 4.7.1 in \cite{ch:Rassoul-Agha}). %HERE
On the other hand, it is believed that the set $D$ is a purely technical condition --- i.e., that the statement of the theorem should hold with $D = [0, 2\pi)$. See Section \ref{sec:lnredux}, where we discuss improvements to Theorem \ref{thm:lnunique}.

The proof of Theorem \ref{thm:lnunique} rests on the following weaker statement, which we give as a lemma.
\begin{lemma}\label{lem:lnlem}
	There exists a deterministic set $D \subseteq [0, 2\pi)$ of Lebesgue measure $2 \pi$ such that the following holds. For any $\theta \in D$,  
\[\prob\left(\text{there are distinct $\theta$-directed unigeodesics } \gamma_1,\, \gamma_2 \text{ starting from 0}\right) = 0. \]
\end{lemma}
\begin{proof}
	Consider the random set \[\Phi = \{ \theta \in [0, 2 \pi): \text{there are distinct } \theta-\text{directed unigeodesics starting from }0\}.\] The main goal of the argument is to show that $\Phi$ is a.s. countable by giving an explicit algorithm for indexing $\Phi$. 

Consider a direction $\theta \in \Phi$ and choose two distinct $\theta$-directed unigeodesics $\gamma_1, \, \gamma_2$ starting from $0$. Note that $\gamma_1$ and $\gamma_2$ must bifurcate at some vertex $z \in \Z^2$. Consider the concatenation of the portions of $\gamma_1$ and $\gamma_2$ from $z$ onward as a doubly-infinite self-avoiding path $\Gamma$. The path $\Gamma$ divides the plane into two components (that is, $\mathbb{R}^2 \setminus \Gamma$ has two connected components), one of which is asymptotically contained in small sectors around the direction $\theta$. We call this the region ``between'' $\gamma_1$ and $\gamma_2$. Moreover, one of the paths $\gamma_1,$ $\gamma_2$ is asymptotically counterclockwise of the other. Without loss of generality, we assume it is $\gamma_1$.

Suppose that the first edge of $\gamma_1$ after $z$ is $(z, y_1)$, and similarly let $(z, y_2)$ denote the corresponding edge of $\gamma_2$. We define algorithms producing two distinct $\theta$-directed unigeodesics $\alpha_1(z, y_1, \cT(0))$ and $\alpha_2(z, y_2, \cT(0))$. The first edge (considered as directed for purposes of the algorithm) of $\alpha_1$ is $f_1 = (z, y_1)$. To find the next edge, consider the set of all outgoing directed edges in $\cT(0)$ which contain and are directed away from $y_1$, and which are part of some infinite geodesic in $\cT(0)$. Out of all such edges, choose the clockwise-most; this is the second edge $f_2$ of $\alpha_1(z, y_1, \cT(0))$. From the ``far'' endpoint of $f_2$ (i.e., the endpoint other than $y_1$), we repeat the previous step.

The sequence of edges so produced is an infinite path in $\cT(0)$ and hence an infinite geodesic. The construction guarantees that $\alpha_1$ contains only edges of $\gamma_1$ and edges in the region between $\gamma_1$ and $\gamma_2$. Indeed, if $\alpha_1$ ever branches off of $\gamma_1$ and then re-intersects $\gamma_1$ or $\gamma_2$, there would be a loop in $\cT(0)$, which is forbidden by Lemma \ref{lem:geotree}.  Therefore, $\alpha_1$ is also $\theta$-directed. Similarly, $\alpha_2$ (defined via a similar, counter-clockwise search algorithm) is also $\theta$-directed. By a similar reasoning to the above, $\alpha_1$ and $\alpha_2$ cannot intersect, since this would again produce a loop. % in $\cT(0)$.

In particular, on the event $\Omega_u$, there exists some triple of vertices $(z, y_1, y_2)$ for which the corresponding $\alpha_1$ and $\alpha_2$ geodesics are distinct and $\theta$-directed. This gives that $\Phi$ is a.s. countable.

We now conclude the proof. Since $\Phi$ is a.s. countable (and letting $\mathrm{Leb}$ denote Lebesgue measure),
\begin{equation}
  \label{eq:lebesgue}
  \E \int_{0}^{2 \pi} \mathbf{1}_{\Phi}(u) \mathrm{d}u  = \E \mathrm{Leb}(\Phi) = 0\ .
\end{equation}
Interchanging the order of integration in \eqref{eq:lebesgue}, we get that $\prob(u \in \Phi) = 0$ for (Lebesgue-) almost every $u$ in $[0, 2 \pi)$.
\end{proof}

The proof of Theorem \ref{thm:lnunique} now proceeds by showing that the existence of parallel geodesics would imply that certain geodesic structures appear, then arguing that translation-invariance of the model prohibits these structures from existing. The way translation-invariance enters in the proof of Licea-Newman is via a ``lack of space'' argument due originally to Burton and Keane \cite{BK89}. We replace this step with an argument originally appearing in \cite{DH14}, which uses a 
\index{mass transport}%
``mass transport'' argument as encapsulated in the following lemma.
\begin{lemma}\label{lem:mass}
  Let $f(\cdot, \cdot): \, \Z^2 \times \Z^2 \to [0, \infty]$. Suppose that $f$ is translation-invariant in the sense that $f(x + z, y + z) = f(x, y)$ for all $x,\,y,\,z \in \Z^2$. Then
\begin{equation}\label{eq:mass}
\sum_{y \in \Z^2} f(x, y) = \sum_{y \in \Z^2} f(y, x) \quad \text{for all } x \in \Z^2\ . \end{equation}
\end{lemma}
\begin{proof}
  Note that $\sum_{y \in \Z^2} f(x, y) = \sum_{z \in \Z^2} f(x, x+z)$. Moreover, translation-invariance gives $f(x, x+z) = f(x-z, x)$. Thus,
\[\sum_{y \in \Z^2} f(x, y) = \sum_{z \in \Z^2} f(x-z, x) = \sum_{y \in \Z^2} f(y,x)\ .\qedhere \]
\end{proof}
Lemma \ref{lem:mass} can be used as follows. Suppose we define a ``mass transport rule'' assigning a (random) nonnegative number $m(x,y)$ to each ordered pair $(x, y) \in \Z^2 \times \Z^2$. This is thought of as an amount of mass moved from vertex $x$ to vertex $y$. If $\E m(x + z ,y + z ) = \E m(x,y)$ for all vertices $x$, $y$, and $z$, then we can apply the lemma with $f = \E m$. In our application of this technique, we argue that if Theorem \ref{thm:lnunique} did not hold, then one could define such an $m$ in a way that $\sum_y m(x,y) \leq 1$ for all $x$, but such that $\sum_{y} \E m(y,x) = \infty$ for all $x$. This contradicts the equality \eqref{eq:mass}.

\begin{proof}[Sketch of proof of Theorem \ref{thm:lnunique}]
  The proof is easiest in the case that the edge weights have unbounded support --- that is, $\prob(t_e > \lambda) > 0$ for all $\lambda > 0$. We will give the proof assuming this in conjunction with our standing assumptions; we direct the reader to the original paper for the modifications that must be made if the weights have bounded support.

  The set $D$ of the theorem is the same as the $D$ from Lemma \ref{lem:lnlem}. For simplicity of notation, we assume that $e_1 \in D$ and show the statement for $\theta = e_1$. Let $F'$ be the event that, from each $x \in \Z^2,$ there is at most one unigeodesic in $\cT(x)$ having direction $e_1$; set $F = F' \cap \Omega_u$. It follows from Lemma \ref{lem:lnlem} (under the assumption $e_1 \in D$) that $\prob(F) = 1$.

The proof proceeds by contradiction. Assume that with positive probability there exist two distinct $e_1$-directed unigeodesics. By considering these geodesics after their last intersection, we may assume they are disjoint. Since they are $e_1$-directed, they eventually intersect each line $L_n = \{x \in \mathbb{R}^2: \, x \cdot e_2 = n\}$ for $n$ large. They also eventually pass each $L_n$ and never return; in particular, the geodesics have a last intersection point with each $L_n$ for $n$ large. Thus, there is some $n$ and some pair $x, y \in L_n$ such that $\prob(E_n(x,y)) > 0$, where
  \begin{align*}
    E_n(x,y) = \{&\text{there exist disjoint } e_1\text{-directed unigeodesics from } x,\, y\\
    &\text{ intersecting } L_n \text{ only at } x,\, y \text{ respectively}\}\ ,
  \end{align*}
  and we write $E(\cdot, \cdot) := E_0(\cdot, \cdot)$.

By translation-invariance of the model, we get $\prob(E(0, z_0)) > 0$ for some particular vertex $z_0$  on the positive $e_2$-axis. We have found two disjoint geodesics in the right half-plane; we now find a third. Letting $E'(x,y,z)$ (where $x,y,z \in L_0$) be the event that there are disjoint $e_1$-directed unigeodesics starting at and touching $L_0$ only at $x,$ $y$, and $z$ respectively, we claim that $\prob(E'(0, z_1, z_2 )) > 0$ for some vertices $z_1$ and $z_2$ of $L_0$ satisfying $0 < z_1 \cdot e_2 < z_2 \cdot e_2$. 

We first note that we can find an integer $m > 1$ such that \[\prob(E(0, z_0) \cap E(m z_0 , (m+1) z_0) ) > 0.\] This follows from the ergodic theorem: we have the almost sure and $L^2$- convergence
\[ \frac{1}{M} \sum_{k=0}^{M-1} \mathbf{1}_{E( k z_0, (k+1)z_0)} \rightarrow \prob(E(0, z_0))\ . \]
Thus $\prob(E(k_1 z_0, (k_1+1) z_0) \cap E( k_2 z_0, (k_2+1) z_0)) > 0$ for some $k_1 < k_2 - 1$ and hence (using translation invariance) $\prob(E(0, z_0) \cap E((k_2 - k_1) z_0, (k_2 - k_1 + 1) z_0)) > 0$.

On the event $F$, for each $x \in \Z^2$, let $\Gamma(x)$ denote the unique $e_1$-directed unigeodesic in $\cT(x)$ (if it exists). On $F$, for any $x,y$ such that $\Gamma(x) \cap \Gamma(y) \neq \varnothing$, the geodesics $\Gamma(x)$ and $\Gamma(y)$ are not distinct (they coalesce). Indeed, if $\Gamma(x)$ and $\Gamma(y)$ intersected and subsequently separated, there would be some $z \in \Gamma(x) \cap \Gamma(y)$ such that $\cT(z)$ has two distinct $e_1$-directed unigeodesics, an impossibility on $F$.

Take $m$ as above. We now show that $\prob(E'(0, z_0, (m+1) z_0)) > 0$, which follows from
\begin{equation}\label{eq:fourtothree}
E'(0, z_0, (m+1)z_0) \supseteq E(0, z_0) \cap E(m z_0, (m+1) z_0) \cap F\ . 
\end{equation}
 To see \eqref{eq:fourtothree} holds, we argue that on the event on the right-hand side, $\Gamma(0)$ cannot intersect $\Gamma(z_0), \Gamma(m z_0),$ or $\Gamma((m+1) z_0)$.  The fact that $\Gamma(0)$ cannot intersect $\Gamma(m z_0)$ is a consequence of the occurrence of $E(0,z_0)$.  If $\Gamma( z_0)$ did not intersect $\Gamma(m z_0)$, then it would be clear that $\Gamma(0)$ cannot intersect $\Gamma(m z_0)$ either, since it cannot intersect $L_0$ or $\Gamma(z_0)$. On the other hand, if $\Gamma(m z_0)$ and $\Gamma(z_0)$ intersect, then they must merge; this also prevents $\Gamma(0)$ from intersecting $\Gamma(m z_0)$ (since it would then merge with $\Gamma( z_0)$). Similar arguments hold for the other geodesics; out of $\Gamma(0),\,\Gamma(z_0),\, \Gamma(mz_0),$ and $\Gamma((m+1)z_0)$, only $\Gamma(z_0)$ and $\Gamma(m z_0)$ may merge. This establishes \eqref{eq:fourtothree}; see Figure \ref{fig:lnthree}.
\begin{figure}\label{fig:lnthree}
  \caption{The event $E'(0, z_0, (m+1) z_0)$ is implied by $E(0, z_0) \cap E(m z_0, (m+1) z_0)$ and uniqueness of $e_1$-directed geodesics. Although $\Gamma(z_0)$ and $\Gamma(m z_0)$ may coalesce (as depicted), $\Gamma(0)$ cannot intersect $\Gamma((m+1) z_0)$. $\Gamma(z_0)$ is drawn dashed to represent its role as part of a  barrier insulating $\Gamma((m+1)z_0)$ from $\Gamma(0)$; the vertical axis (representing $L_0$) completes the barrier. The edge set $\mathfrak{E}$ is depicted as the shaded gray rectangle. }
  \centering
    \includegraphics[width=0.5\textwidth]{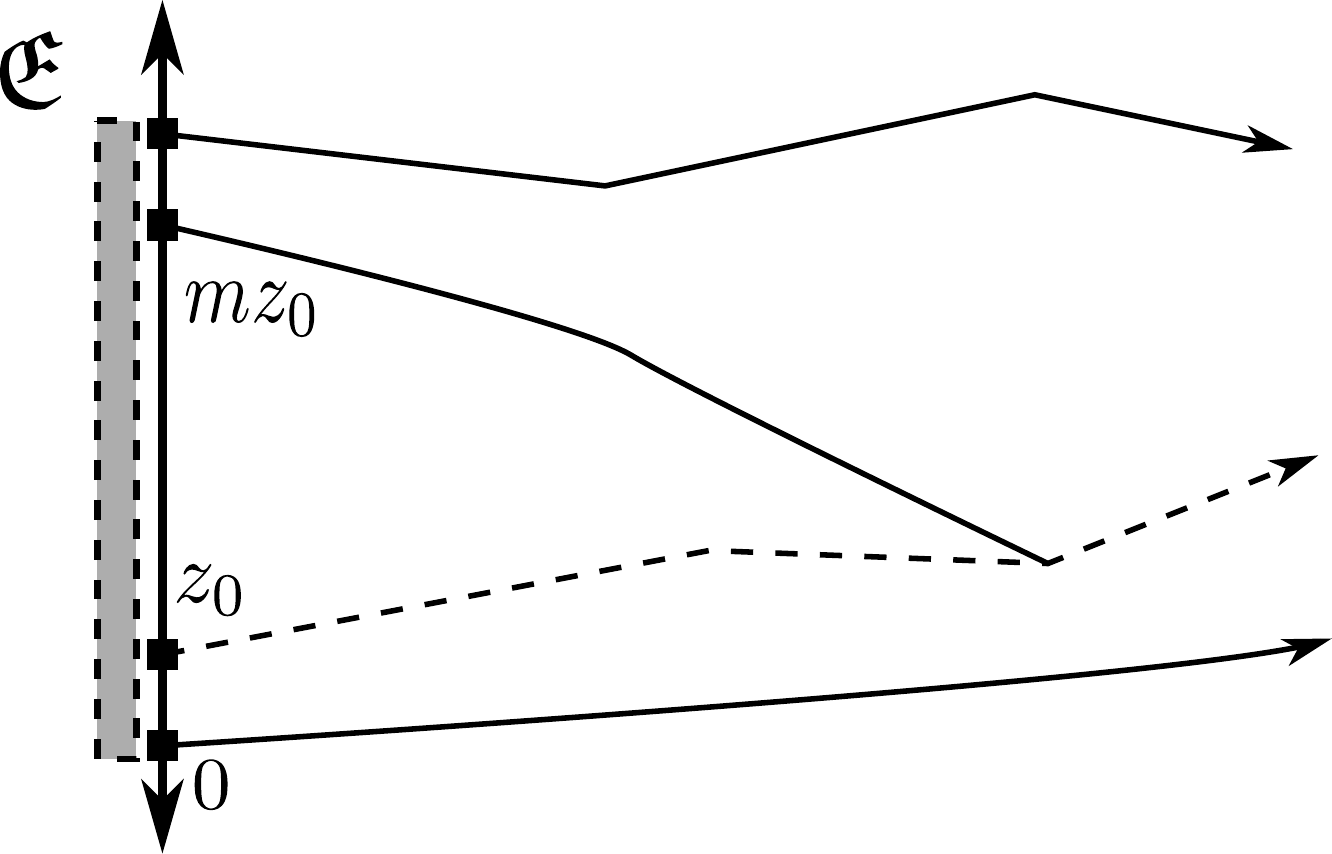}
\end{figure}

We now define the event $E'' = E''(0, z_0, (m+1) z_0)$ to be the subevent of $E'(0, z_0, (m+1)z_0)$ on which no geodesic $\Gamma(y)$ intersects $\Gamma(z_0)$, for any $y$ satisfying either 
\begin{itemize}
\item $y \cdot e_1 < 0$ or
\item $y \cdot e_1 = 0$ with either $y \cdot e_2 > (m+1) z_0 \cdot e_2$ or $y \cdot e_2 < 0$.
\end{itemize}
We argue $\prob(E'') > 0$. Consider the set $\mathfrak{E}$ of edges touching $L_0$ between $0$ and $(m+1)z_0$:
\[\mathfrak{E} = \{\{x - e_1, x\}:\, x \cdot e_1 = 0, \, 0 \leq x \cdot e_2 \leq (m+1) z_0 \cdot e_2\} \ . \]
Let $A_\lambda$ denote the event that, for each $f \in \mathfrak{E}$, there is a path $\gamma$ connecting the endpoints of $f$ which contains no edge of $\mathfrak{E}$ and which satisfies $T(\gamma) < \lambda$. Clearly $\prob(A_\lambda) \rightarrow 1$ as $\lambda \rightarrow \infty$, so we can fix a value of $\lambda > 0$ such that $\prob(A_\lambda \cap E'(0, z_0, (m+1)z_0)) > 0$.

Now we note that $A_\lambda \cap E'$ is increasing in the edge weights in $\mathfrak{E}$. In other words, given an outcome $(t_e) \in A_{\lambda} \cap E'$, if $(t_e') \in F$ is an outcome satisfying $t_e' = t_e$ for $e \notin \mathfrak{E}$ and $t_e' \geq t_e$ for $e \in \mathfrak{E}$, then $(t_e') \in A_{\lambda} \cap E'$ as well. The point is that if we increase edge weights which do not lie on some path $\Gamma$ which was a geodesic in the original edge-weight configuration, then $\Gamma$ is also a geodesic in the new configuration. 

We use this observation to show $\prob(E'') > 0$ by ``modifying'' outcomes in $A_\lambda \cap E'$ so that each $f \in \mathfrak{E}$ has $t_f > \lambda$ and so that the resulting outcome remains in $A_\lambda \cap E'$. Implementing this modification in a way which is guaranteed to produce a measurable set requires some care; see \cite[Lemma 3.1]{LN96}. The resulting event can be shown to have positive probability by the fact that $\prob(t_e > \lambda) > 0$ for every $\lambda$. 

The outcomes produced by the modification are in $E''$ because the geodesic $\Gamma(z_0)$ is ``shielded'' by $\mathfrak{E}$ and the surrounding geodesics. Indeed, if a $\Gamma(y)$ as in the definition of $E''$ intersected $\Gamma(z_0)$, it must first either cross $\Gamma(0)$, $\Gamma((m+1) z_0)$, or an edge of $\mathfrak{E}$. The first two possibilities are ruled out by the definition of $E'$ and the second is ruled out by the fact that $A_\lambda$ occurs but each edge of $\mathfrak{E}$ has weight $>\lambda$, so no geodesic can use an edge of $\mathfrak{E}$.

We now use the conclusion that $\prob(E'') > 0$ to derive a contradiction. Let $\prec$ denote the lexicographic ordering where $x \prec y$ if $x \cdot e_1 \leq y \cdot e_1$ or if $x \cdot e_1 = y \cdot e_1$ and $x \cdot e_2 \leq y \cdot e_2$. Define the function $m(x, y)$ for $x, y \in \Z^2$ as follows:
\begin{align*}
m(x,y) :=
\begin{cases}
  1, \quad &x \in \Gamma(y) \text{ and } x \notin \Gamma(z) \text{ for any } z \neq y \text{ satisfying } z \prec y;\\
  0, \quad &\text{otherwise.}
\end{cases}
\end{align*}
In other words, $m(x,y) = 1$ if $y$ is the $\prec$-minimal element of $\Z^2$ such that $x \in \Gamma(y)$. Note that $\sum_{y} m(x,y) \leq 1$ almost surely for any $x$, since there can be at most one such minimal ancestor for any $x$.

On the other hand, on the event $E'$, there is some $x \in L_0$ between $0$ and $(m+1)z_0$ such that $\sum_y m(y, x) = \infty$. Either $x = z_0$, or $x$ is some other vertex of $L_0$ between $0$ and $(m+1) z_0$ having $z_0 \in \Gamma(x)$. Using $\prob(E') > 0$ and translation invariance, $\sum_y \E m(y,0) = \infty$, and using the almost sure upper bound above gives $\sum_y \E m(0,y) \leq 1$. Setting $f(x,y) = \E m(x,y)$, translation-invariance of the model gives $f(x+z, y+z) = f(x,y)$. Applying Lemma \ref{lem:mass}, these estimates give a contradiction and complete the proof of the theorem.
\end{proof}

\section{Showing $\cN \geq 2$}
%H\"aggstr\"om -- Pemantle and Gar\'et -- Marchand}
The first unconditional result showing that $\prob(\cN \geq 2) > 0$ is due to H\"aggstr\"om and Pemantle \cite{HP98} in the case that $d = 2$ and $F$ is an exponential distribution. We will not discuss the details of their proof, since it is very specialized to that setting. We shall instead show a more general result due to both Garet-Marchand \cite{GM05} and Hoffman \cite{H05}, which applies to all $\Z^d$:
\begin{theorem}[Garet-Marchand \cite{GM05}, Hoffman \cite{H05}]\label{thm:coex}
  Under our standing assumptions, $\prob(\cN \geq 2) > 0$.
\end{theorem}
We will give an adaptation of the proof of Garet-Marchand in this section, with an eye towards subsequently introducing the Busemann function methods of Hoffman.

Our main goal is to show that with positive probability there exist sequences $m_k, \, n_k \rightarrow \infty$ such that $\Gamma_0 := \lim_k G(0, - n_k e_1)$ and $\Gamma_1:= \lim_k G(e_1, m_k e_1)$ are distinct. We will actually show something more general, which illustrates an ``infection model'' principle for finding multiple geodesics; see Lemma \ref{lem:generalco} below. Let us define, for each $z \in \Z^d$, $B_z(x,y) := T(x, z) - T(y, z)$. Note that $B_z(x,y) = -B_z(y,x)$.

We shall see we can find distinct unigeodesics $\Gamma_0$ and $\Gamma_1$ if we can show $B_{ne_1}(0, e_1)$ is ``positively biased'' for large positive $n$. In fact, it is easier to show this with $e_1$ replaced by $k e_1$, because the control we get over $B$ will be asymptotic.

\begin{lemma}\label{lem:generalco}
 Fix a finite set of vertices $\{x_1, \ldots, x_k\} \subseteq \Z^d$. 
    Let us define $G_i,\, 1 \leq i \leq k$ by
  \[G_i:= \{z \in \Z^d: B_z(x_i, x_j) < 0 \quad \text{for all } j \neq i \} \]
--- i.e., $G_i$ is the set of vertices closest to $x_i$ in the $T$-metric. On the event that $|G_i| = \infty$ for all $1 \leq i \leq k$, we have $\cN \geq k$.
\end{lemma}

\begin{corollary}
Assume that for some $\ell > 0$,
\label{cor:halfinf}
\begin{equation}
  \label{eq:halfinf}
  \prob\left(\left\{\limsup_{n \rightarrow \infty} B_{ne_1}(0, \ell e_1) > 0\right\} \cap \left\{\limsup_{n \rightarrow \infty} B_{-ne_1}(\ell e_1, 0) > 0\right\} \right) > 0\ .
\end{equation}
Then $\prob(\cN \geq 2) > 0$.
\end{corollary}
\begin{proof}

Letting $x_1 = 0,\,x_2 = \ell e_1$, on  the event in \eqref{eq:halfinf} we have $|G_i| = \infty,\, i = 1, \, 2$. The result now follows from Lemma \ref{lem:generalco}\ .
\end{proof}

\begin{proof}[Proof of Lemma \ref{lem:generalco}]
  As a preliminary, we observe that $G_i$ and $G_j$ are disjoint by construction, and that with probability one $x_i \in G_i$. We will show that for each $i$ there is an infinite path $\Gamma_i$ of vertices of $G_i$, having $x_i$ as its initial vertex, which is actually a unigeodesic.

  For each $i$, choose a sequence $(z_k(i))_k$ of distinct elements of $G_i$ such that $\Gamma_i:= \lim_k G(x_i, z_k(i))$ exists; that this is possible follows via the 
  \index{diagonal argument}%
  diagonal argument of the proof of Proposition \ref{prop:onegeo}. We claim that each vertex of $\Gamma_i$ is an element of $G_i$. If there were some $z \in \Gamma_i$ with $z \notin G_i$, then for some $j$, $T(x_i, z) \geq T(x_j, z)$. But since $z \in G(x_i, z_k(i))$ for some $k$, we have $T(x_i, z_k(i)) = T(x_i, z) + T(z, z_k(i)$).

On the other hand, subadditivity gives
\begin{align*}
  T(x_j, z_k(i)) \leq T(x_j, z) + T(z, z_k(i)) &\leq T(x_i, z) + T(z, z_k(i))= T(x_i, z_k(i))\ ,
\end{align*}
so $z_k(i) \notin G_i$. This is a contradiction.
\end{proof}

The structure of Lemma \ref{lem:generalco} is based on viewing the FPP model as defining a 
\index{competition model}%
competition model ($G_i$ is the set of sites ``infected'' by a disease with initial patient located at site $x_i$; sites can be infected only by a single type of disease). This perspective on the geodesic question appeared in the original paper of H\"aggstr\"om-Pemantle. The proof of Lemma \ref{lem:generalco} amounts to showing that each infection spreads along FPP geodesics.

\begin{proof}[Proof of Theorem \ref{thm:coex}]
  We use contradiction to show that hypothesis \eqref{eq:halfinf} of Corollary \ref{cor:halfinf} holds. So choose $\ell$ so large that $\bfE T(0, \ell e_1) < 3 \ell g(e_1)/2$ (this choice will become clear in a moment), and defining
\[  A_0:= \left\{\limsup_{n \rightarrow \infty} B_{ne_1}(0, \ell e_1) > 0\right\} \quad \text{and}\quad A_\ell:= \left\{\limsup_{n \rightarrow \infty} B_{-ne_1}(\ell e_1, 0) > 0\right\}\ , \]
assume that $\prob(A_0 \cap A_\ell) = 0$. We begin with a simple bound: note that the triangle inequality implies that $T(0, n e_1) \leq T(0, \ell e_1) + T(\ell e_1, n e_1)$, so $B_{n e_1}(0, \ell e_1) \leq T(0, \ell e_1)$ for all $n$. (A similar argument gives $B_{ne_1}(0, \ell e_1) \geq - T(0, \ell e_1)$ and in particular that $B_{ne_1}(0, \ell e_1)$ is integrable).

We can of course bound $B_{n e_1}(0, \ell e_1)$ by $0$ on the event $A_0^c$, and give a similar bound for $B_{-ne_1}(\ell e_1, 0)$. In particular,
\begin{align}
  \limsup_{n \rightarrow \infty} B_{ne_1}(0, \ell e_1) +  \limsup_{m \rightarrow \infty}B_{-me_1}(\ell e_1,0)  &\leq T(0, \ell e_1) \mathbf{1}_{A_0} + T(\ell e_1, 0) \mathbf{1}_{A_\ell} \nonumber\\
  &= T(0, \ell e_1)[\mathbf{1}_{A_0} + \mathbf{1}_{A_\ell}]\ . \label{eq:busegm}
\end{align}
Since $\prob(A_0 \cap A_{\ell}) = 0$, we have $\mathbf{1}_{A_0} + \mathbf{1}_{A_\ell} \leq 1$ almost surely. Using this and our choice of $\ell$, we take expectations in \eqref{eq:busegm} to see
\begin{equation}
\label{eq:busesmall}
\bfE\left[\limsup_{n \rightarrow \infty} B_{ne_1}(0, \ell e_1) +  \limsup_{m \rightarrow \infty}B_{-me_1}(\ell e_1,0)\right] \leq \bfE T(0, \ell e_1) < 3 \ell g(e_1)/2\ . \end{equation}

So far we have not done anything to show that there is anything wrong with what has been derived. The key step is to show that $B_{n e_1} (0, \ell e_1)$ is typically of order $(\ell - \varepsilon)g(e_1)$ for a large density of $n$. Combined with a corresponding bound for the other term, we show that the left-hand side of \eqref{eq:busesmall} is of order $(2 \ell - \varepsilon) g(e_1)$, in contradiction to the bound on the right-hand side. This contradicts the assumption that $\prob(A_0 \cap A_\ell) = 0$.

The estimate described above is shown following an averaging trick; a more sophisticated version of this will appear in the proof of Theorem \ref{thm:buseasymp} below. 
Considering multiples of $\ell e_1$, we consider terms of the form
\begin{equation}
  \label{eq:gmt}
  B_{K \ell e_1}(0, \ell e_1) = T(0, K \ell e_1) - T(\ell e_1, K \ell e_1).
\end{equation}
Note that the model is invariant under lattice shifts, so
\begin{equation}
  \label{eq:gmavg}
  \begin{split}
  \bfE B_{K \ell e_1}(0, \ell e_1) 
  &= \bfE T(0, K \ell e_1) - \bfE T(\ell e_1, K \ell e_1) \\
  &= \bfE T(0, K \ell e_1) - \bfE T(0, (K - 1)\ell e_1)\ .
  \end{split}
\end{equation}
This shifts the perspective of the infection process and allows us to apply shape theorem results.

Applying \eqref{eq:gmavg} to terms of the type \eqref{eq:gmt} gives
\begin{equation*}
\bfE \sum_{K=1}^n B_{K \ell e_1}(0, \ell e_1) = \bfE T(0, n \ell e_1)\ . \end{equation*}
Dividing by $n$ and using the shape theorem gives that
\begin{equation}
  \label{eq:gmavg2}
    \lim_{n \rightarrow \infty} \frac{1}{n} \sum_{K=1}^n\bfE B_{K \ell e_1}(0, \ell e_1) = g(\ell e_1) = \ell g(e_1).
\end{equation}
In particular, \eqref{eq:gmavg2} gives the bound $\limsup_{n \rightarrow \infty} \bfE B_{n e_1} (0, \ell e_1) \geq \ell g(e_1)$, with a similar bound for $\limsup_{m \rightarrow \infty} B_{-m e_1} (\ell e_1, 0)$.
Since $|B_{z}(0, \ell e_1)| \leq  T(0, \ell e_1)$, it follows (e.g. by Fatou's lemma) that
\[\bfE \limsup_{n \rightarrow \infty} B_{n e_1}(0, \ell e_1) \geq \limsup_{n \rightarrow \infty} \bfE B_{n e_1}(0, \ell e_1)  \geq \ell g(e_1)\ .\]
Applying this in \eqref{eq:busesmall} (with the corresponding bound for the other term) completes the proof.
\end{proof}

\section{Hoffman's method and Busemann functions}
One can think of the function $B_{n e_1}(x,y)$ from the preceding section as a sort of ``relative distance to infinity in direction $e_1$'' between points $x$ and $y$, for $n$ very large. If one could show that the limit $B:= \lim_n B_{n e_1}$ existed, one could work directly with the object $B$ in Garet-Marchand's proof, which would make certain results more natural. For instance, a part of their proof amounts to showing that $\bfE B_{n e_1}(0,e_1) \approx g(e_1)$ for a large density of but potentially not all $n$ --- given existence of $B$, we could hope instead to show $\bfE B(0, e_1) = g(e_1)$ and avoid working with subsequences.

We will come back to this perspective in the $e_1$-direction and other fixed directions later, though the idea of working directly with a limiting difference of passage times will be a key point of this section. Our main goal here will be to outline Hoffman's approach to Theorem \ref{thm:coex}. The key idea is to notice that there is a natural object similar to $B_{n e_1}$ for which limits are easily shown to exist and which has similar implications for geodesics. This object is called a 
\index{Busemann function}%
{\bf Busemann function} and was originally introduced by Busemann in the study of (nonrandom) metric geometry.

\begin{definition}
  Consider a fixed realization of edge weights and suppose $\gamma = (z_1, z_2, \ldots)$ is a unigeodesic in this configuration. The Busemann function $B_\gamma(\cdot,\cdot)$ is defined by 
\begin{equation}
\label{eq:buselim}
B_\gamma(x,y):= \lim_{n\rightarrow \infty}[T(x, z_n) - T(y, z_n)]\ .
\end{equation}
\end{definition}
\begin{clam}\label{clam:buselim}
  The limit \eqref{eq:buselim} exists for any configuration of non-negative edge weights, each pair $x,\, y$, and each unigeodesic $\gamma$.
\end{clam}
\begin{proof}
  Note that the triangle inequality gives that the terms on the right-hand side of \eqref{eq:buselim} are bounded in magnitude by $T(x,y)$. For instance, $T(x, z_n) \leq T(x, y) + T(y, z_n)$ gives $T(x, z_n) - T(y, z_n) \leq T(x,y)$. This implies that if we can show the monotonicity (in $n$) of these terms for a fixed pair $x,\,y$, we will have established existence of the limit.

  We first consider the case $y = z_1$. By geodesicity of $\gamma$, we have for each $1<i < n$ that $T(z_1, z_n) = T(z_1, z_i) + T(z_i, z_n)$. On the other hand, subadditivity of course yields $T(x, z_n) \leq T(x, z_i) + T(z_i, z_n)$. Thus,
  \begin{align*}
    T(x, z_n) - T(z_1, z_n) 
    &\leq T(x, z_i) + T(z_i, z_n) - [T(z_1, z_i) + T(z_i, z_n)] \\
    &= T(x, z_i) - T(z_1, z_i)\ .
  \end{align*}
Therefore, the terms on the right-hand side of \eqref{eq:buselim} are nonincreasing in $n$, which (combined with the boundedness already established) shows existence of the limit.

Now we consider the case $y \neq z_1$. We can rewrite the typical term on the right-hand side of \eqref{eq:buselim} as
\begin{equation}
\label{eq:busetel}
T(x,z_n) - T(y, z_n) = [T(x, z_n) - T(z_1, z_n)] - [T(y, z_n) - T(z_1, z_n)]\ . \end{equation}
Since the limits of each of the bracketed terms on the right-hand side of \eqref{eq:busetel} exist, so does the limit of the left-hand side.
\end{proof}

Being a difference of passage times, $B_\gamma$ inherits several properties which are key to its analysis.
\begin{lemma}\label{lem:buseprops}
  On the event $\Omega_u$, the following hold for each unigeodesic $\gamma = (z_1, z_2, \ldots)$ and all vertices $x,\,y,\,z \in \Z^d$:
  \begin{enumerate}
    \item[{\rm(1)}] $|B_\gamma(x,y)| \leq T(x,y)$ ;
  \item[{\rm(2)}] $B_\gamma (x,y) + B_\gamma (y,z) = B_\gamma (x,z)$ ;
      \item[{\rm(3)}] $B_\gamma(y,x) = -B_\gamma(x,y)$;
      \item[{\rm(4)}] If $x = z_i$ and $y = z_j$ with $i < j$ (so $x$ appears before $y$ in the unigeodesic), then $B_\gamma (x,y) = T(x,y)$.
        \item[{\rm(5)}] If $\widetilde \gamma$ is another unigeodesic which is not distinct from $\gamma$, then $B_{\gamma} = B_{\widetilde \gamma}$.
  \end{enumerate}
\end{lemma}
\begin{proof}
  Property (1) already was shown during the proof of Claim \ref{clam:buselim} to establish boundedness of the sequence appearing in \eqref{eq:buselim}. Properties (2)--(3) follow easily from the fact that $B_\gamma$ is defined as a difference: for instance, to see (3), write
  \begin{equation}
    \label{eq:diff}
    T(y, z_n) - T(x, z_n) = -[T(x,z_n) - T(y,z_n)]
  \end{equation}
  and take the limit of \eqref{eq:diff} as $n \rightarrow \infty$.

To see (4), note that since $\gamma$ is a geodesic, we have for $n > j$:
\[T(x, z_n) = T(z_i, z_n) =  T(z_i, z_j) + T(z_j, z_n) = T(x,y) + T(y, z_n)\ . \]
Subtracting $T(y, z_n)$ from both sides and taking limits establishes (4).

Property (5) is where we use  $\Omega_u$. Let $\widetilde \gamma  = (z_1', z_2',\, \ldots)$. By Proposition \ref{prop:distinct}, we see that the unigeodesics coalesce, so there are some $i$ and $j$ such that $z_i = z_j',\, z_{i+1} = z_{j+1}', \ldots$. In particular, the limit appearing in \eqref{eq:buselim} is unchanged if we replace $z_n$ by $z_n'$.
\end{proof}

We now are almost equipped to give Hoffman's version of a proof of Theorem \ref{thm:coex}. Since the argument involves applying an ergodic theorem, we take a moment here to define the operators which shift configurations by integer vectors.
\begin{definition}\label{defin:trans}
  Let $z \in \Z^d$, and let $\omega$ be a realization of the edge weights $(t_e)$. We define the 
  \index{shift}%
  shift $\theta_z$ to be the operator which acts on $\omega$, producing a new configuration $\theta_z \omega$, as follows: $\theta_z \omega$ is the realization of edge weights $(t_e')$, where $t'_{\{x,y\}} = t_{\{x + z,\, y + z\}}.$
\end{definition}
\begin{proof}[Proof of Theorem \ref{thm:coex}]
  Assume that $\prob(\cN = 1) = 1$. Using Proposition \ref{prop:onegeo}, we see that we are guaranteed the existence of a unigeodesic $\gamma$ beginning from $0$. This $\gamma$ must be a.s. unique, in the strong sense that there cannot exist an infinite path from $0$ which is a unigeodesic, other than $\gamma$. Indeed, if there were another unigeodesic $\widetilde \gamma$ starting from $0$, then by the non-distinctness assumption there must be some $0 \neq z \in \gamma \cap \widetilde \gamma$ beyond which the two geodesics are the same. But by a.s. uniqueness of finite geodesics, the subpaths of both $\gamma$ and $\widetilde \gamma$ from $0$ to $z$ must be identical, so these two unigeodesics must be exactly the same path.

The translation-invariance of the model gives that we can apply Proposition \ref{prop:onegeo} to find a unigeodesic $\Gamma(x)$ from each $x$ in $\Z^d$. Since $\cN = 1$, Proposition \ref{prop:distinct} gives us that each $\Gamma(x)$ must coalesce with $\gamma$. Let us define $B = B_\gamma$, where $\gamma$ is as in the preceding paragraph. The proof of Proposition \ref{prop:onegeo} in fact shows that every subsequential limit of finite geodesics from any initial vertex $x$ must produce $\Gamma(x)$, which coalesces with $\gamma$ --- so $B$ is in some sense a quite explicit function of the edge weights, and it is not hard to see that this implies the measurability of $B$. The property (1) of Lemma \ref{lem:buseprops} and the integrability of $T$ gives that $B(x,y)$ has finite first moment for all $x,\,y$.

In fact, $B$ is translation-covariant, in the following sense. Writing explicitly the dependence of each object on the configuration $\omega$, our assumption that $\cN = 1$ gives that, almost surely, $\gamma[\theta_z \omega] = \Gamma(z)[\omega] - z$. This implies the translation-covariance of $B$: $B(x,y)[\theta_z \omega] = B(z + x, z + y)[\omega]$. We also note that reflections and rotations which fix $0$ and leave $\Z^d$ invariant rotate / reflect $\gamma$, but that $B$ must be invariant in distribution under these operations (by the invariance of the model and the construction of $B$). So $\bfE B(x,y) = 0$ for all $x$ and $y$.

Fix $0 \neq x \in \Z^d$. Using additivity (property (2) of Lemma \ref{lem:buseprops}), shift-covariance, and the ergodic theorem, we have for almost every $\omega$:
\begin{equation}
  \label{eq:hoferg}
  \begin{split}
  \frac{B(0,nx)[\omega]}{n} 
  &= \frac{1}{n}\sum_{i=1}^n B((i-1)x, ix)[\omega]\\
  &= \frac{1}{n}\sum_{i=1}^n B(0, x)[\theta_x^{i-1} \omega] \stackrel{n \rightarrow \infty}{\longrightarrow} \bfE B(0,x) = 0\ .
  \end{split}
\end{equation}
Thus, $B$ grows sublinearly in each direction with probability one. We claim a stronger statement: analogous to the usual 
\index{shape theorem}%
shape theorem, we can make this sublinearity hold simultaneously in all directions with probability one:
\begin{equation}
  \label{eq:hoferg2}
  \prob\left(\lim_{|x| \rightarrow \infty} \frac{|B(0,x)|}{|x|} =  \limsup_{|x| \rightarrow \infty} \frac{|B(0,x)|}{|x|} = 0 \right) = 1\ .
\end{equation}

The proof of \eqref{eq:hoferg2} from \eqref{eq:hoferg} is ({\it mutatis mutandis}) identical to the proof of the shape theorem (see Theorem 3.1 of \cite{ch:Damron}), so we only recall the main idea here. Fix some configuration $\omega$ such that \eqref{eq:hoferg} holds for each $x \in \Z^d$, and assume for the sake of contradiction that there were some sequence $z_k$ along which $|B(0, z_k)| > \varepsilon |z_k|$ uniformly. Without loss of generality, we may assume that $z_k / |z_k| \rightarrow z$ for some $z$. We can find some $x \in \Z^d$ with $x/|x|$ close to $z$. For each $k$ large, let $n_k$ minimize $|n_k x - z_k|$; by careful choice of $x$, we have $|n_k x - z_k| \ll \min\{n_k |x|, \, |z_k|\}$. In particular, since $|B(0, n_k x) - B(0, z_k)| \leq T(n_k x, z_k)$ (and since this passage time is order $|n_k x - z_k|$), we see that $|B(0, n_k x)|$ is large relative to $n_k |x|$. This is in contradiction to the fact that $B(0, n x) / n|x| \rightarrow 0$ on $\omega$.

On the other hand, we can see that sublinearity of $B$ is absurd by considering property (4) of Lemma \ref{lem:buseprops}. Letting $\gamma = (y_1, y_2, \ldots),$ we have $B(y_i, y_j) = T(y_i, y_j)$ for $i < j$. For $j$ large, we have by the shape theorem that $T(0, y_j) \approx g(y_j)$.
To be precise, fix $\delta > 0$ such that $g(z) > \delta$ for all $z$ with $|z| = 1$. This is possible by the boundedness of the limit shape --- i.e., the fact that passage times asymptotically grow linearly in the Euclidean distance. The shape theorem then implies that with probability one,
\[\limsup_{j \rightarrow \infty} \frac{B(0, y_j)}{|y_j|} = \limsup_{j \rightarrow \infty} \frac{T(0, y_j)}{|y_j|} \geq \delta > 0\ . \]
This is in contradiction to \eqref{eq:hoferg2}, showing that our assumption that $\prob(\cN = 1) = 1$ is false.
\end{proof}

\section{Directedness and Busemann functions}
There is a clear similarity between the Busemann function $B_\gamma$ and the object $B_{ne_1}$ considered in the proof of Garet-Marchand. We will take some time to develop this idea here under a strong assumption (different from (Curve) and (Expm)). We will not push these ideas as far as we could, since we will in the next section give a framework for getting around these sorts of strong assumptions.

Suppose we wish to avoid assuming (Curve), but still believe the results of Theorem \ref{thm:newman} and (for at least some dimensions $d$) Theorem \ref{thm:lnunique} should hold. With this as our guidepost, it seems perhaps reasonable to replace assumption (Curve) with the following:
\begin{assumption}[LimG]\label{asn:limg}
%  \begin{itemize}
% \item[(LimG)] 
For each $x \in \Z^d$, the geodesics $G(x, n e_1)$ have a limit $\Gamma(x)$. Moreover, these geodesics all coalesce.
%  \end{itemize}
\end{assumption}
One immediate consequence of Assumption (LimG) is that the Busemann function $B(x,y):= \lim_n [T(x, n e_1) - T(y, ne_1)]$ exists, and is covariant with respect to translations by $e_1$, similarly to the Busemann function in Hoffman's argument. If we continue taking Theorem \ref{thm:newman} as a goal, one could be led to believe that $\Gamma(x)$ should be directed: $\Theta(\Gamma(x)) = \{e_1\}$, or at least that $\Theta(\Gamma(x)) \neq \bbS^{d-1}$.

We will not try to prove anything as strong as directedness, but instead just the following much weaker claim. In what follows, let $S_\delta$ denote the sector of aperture $\delta$ around $- e_1$:
\[S_\delta = \{z \in \bbS^{d-1}: |z + e_1| < \delta\}\ . \]
\begin{theorem}\label{thm:buseas}
  Assume {\rm(LimG)} {\rm(}along with the standard assumptions of Assumption \ref{asn:std}{\rm)}. Then there is some $\delta$ such that $\Theta(\Gamma(x)) \cap S_\delta = \varnothing$ almost surely.
\end{theorem}

The strategy of the theorem's proof is easy to outline given what we have seen. As in the argument of Garet-Marchand, we can use an ``averaging trick'' to show that $\bfE B(0, - m e_1)$ grows linearly as $-m g(e_1)$ for large $m$. In fact, the translation-covariance of $B$ gives that the growth of $B$ occurs almost surely. On the other hand, property 4 of Lemma \ref{lem:buseprops} holds for $B$, giving that $B(0, \cdot)$ behaves as $T(0, \cdot)$ along $\Gamma(0)$. These asymptotics would conflict if $- m e_1$ were on $\Gamma(0)$, and in fact exclude the geodesic coming within some conical region of the axis.
\begin{proof}[Proof of Theorem \ref{thm:buseas}]
  Let us strengthen the assumptions even further to include boundedness: $\prob(t_e \leq M) = 1$ for some finite $M$; we prove the theorem in this setting for $\Gamma(0)$. Assume for the sake of contradiction that the statement of the theorem fails when $\delta = (g(e_1)/16M \sqrt{d})$. Consider an outcome $\omega$ on which $\Theta(\Gamma(0)) \cap S_\delta \neq \varnothing$ (we will also need to assume that $\omega$ lies in various probability one events on which, e.g., $t_e \leq M$ for all $e$; this will become clear in the course of the proof).

The translation-covariance of $B$ gives, as in \eqref{eq:hoferg},
  \begin{equation}\label{eq:newerg}
\lim_{m \rightarrow \infty} B(0, -m e_1)/m = \bfE B(0, - e_1)\quad \text{a.s. and in } L^1\ .\end{equation}

To compute $\bfE B$, let us as before define $B_n(x,y) := T(x, n e_1) - T(y, n e_1)$. Note that the shape theorem (in fact, the subadditivity used to establish the shape theorem) shows that $\bfE T(0, n e_1) / n \rightarrow g(e_1)$. We use the Garet-Marchand averaging trick:
\begin{align*}
  \frac{\bfE T(0, n e_1)}{n} &= \frac{1}{n}\Big([\bfE T(0, n e_1) - \bfE T(0, (n -1 ) e_1)]\\ 
  &\qquad\quad+ [\bfE T(0, (n-1) e_1) - \bfE T(0, (n -2 ) e_1)] + \ldots + \bfE T(0, e_1)\Big)\\
  &= \frac{1}{n}\left(\bfE B_n(0, e_1) + \bfE B_{n-1} (0, e_1) + \ldots + \bfE T(0, e_1) \right)\ .
\end{align*}
Since $B_n(0, e_1) \rightarrow B(0, e_1)$ almost surely and since $|B_n(0, e_1)| \leq T(0, e_1) \leq M$ almost surely, the typical term above converges to $\bfE B(0, e_1)$. On the other hand, the left-hand side converges to $g(e_1)$, so we see
\begin{equation}
  \label{eq:avgfinal}
  \bfE B(0, -e_1) = -\bfE B(0,e_1) = - g(e_1)\ .
\end{equation}

Combining \eqref{eq:avgfinal} with \eqref{eq:newerg}, we see that $B(0, - m e_1) / m \rightarrow  - g(e_1)$ almost surely. We now move toward a contradiction similarly to Hoffman's argument.  Write $\Gamma(0) = (z_1, z_2, \ldots)$. As in Lemma \ref{lem:buseprops}, we have $B(0, z_i) = T(0, z_i)$ for all $i$. In particular, for all $i$,  $B(0, z_i) \geq 0$. Let $m_i = \lfloor |z_i| \rfloor$.
\begin{equation}
\label{eq:lastbd}
\begin{split}
|B(0, z_i) - B(0, -m_i e_1)| 
&= |B(-m_i e_1, z_i)| \leq  T(- m_i e_1, z_i) \\
&\leq M \|m_i e_1 + z_i\|_1 \\ 
&\leq 2 M \sqrt{d}(|z_i|) | (z_i / |z_i|) + e_1|\ ,  
\end{split}
\end{equation}
when $|z_i| > 1$. 

On the other hand, as $i \rightarrow \infty$, we have 
\[ \lim_{i \to \infty} B(0, -m_i e_1)/ |z_i|  = \lim_{i \to \infty} B(0, -m_i e_1)/ m_i  = -g(e_1)\ .\]
If $i$ is sufficiently large, then $|z_i / |z_i| + e_1| \leq 2\delta =  (g(e_1)/8M\sqrt{d})$ and $B(0, -m_i e_1) \leq - |z_i| g(e_1)/2$. Then the  above implies along with \eqref{eq:lastbd} that
\[ B(0, z_i)/|z_i| \leq -g(e_1)/2 + (2M\sqrt{d})(g(e_1)/8M\sqrt{d}) \leq -g(e_1)/4 . \]
This contradicts the fact that $B(0, z_i) \geq 0$.
\end{proof}

\section[Busemann limits and general directedness]{Busemann subsequential limits and general directedness statements}
We have shown that averaging properties of Busemann functions can be used to control directedness properties of infinite geodesics, and have some idea of how to implement this strategy in practice. Unfortunately, without assumption (LimG), we are lacking a Busemann function and corresponding geodesics on which to run this program. It is obviously reasonable to want to try to construct geodesics without making any unproven assumptions. Our main goal in the remainder of the article is to run a more sophisticated version of last section's argument which allows us to circumvent (LimG). 

Fix some $\zeta \in \partial \cB$. In all our work in this section, we replace (LimG) with the following assumption:
\begin{assumption}[Dif]
%  \begin{itemize}
%  \item[(Dif)]  
\index{limit shape}%
$\partial \cB$ is differentiable at $\zeta$.
%  \end{itemize}
\end{assumption}
Recall the meaning of this statement is that there is a unique supporting hyperplane $H$ for $\cB$ at $\zeta$; see Section 3.1.3 in \cite{ch:Damron}. %HERE
There is a unique $\rho$ such that $H = \{x: \, x \cdot \rho = 1\}$. Recall that $\cB$ is convex by the shape theorem, so (Dif) is a much weaker assumption than it seems at first glance.

Given a $\zeta$ as in (Dif), we can of course define the ``sector of contact'' of $H$ with $\partial \cB$:
\[S = H  \cap \partial \cB\ . \]
The corresponding set of angles is
\begin{equation}
\label{eq:difang}
\Theta_S:= \{z \in \bbS^{d-1}: \, z/g(z) \in S\}\ . \end{equation}
Our main theorem is that, under (Dif), we can produce a geodesic which is directed in a sector no wider than $\Theta_S$.
\begin{theorem}[\cite{DH14}]
\label{thm:dh}
  Assume the standard assumptions and {\rm(Dif)}. Then with probability one, there is an infinite geodesic $\gamma$ from $0$ which is 
  \index{limiting direction}%
  directed in $\Theta_S$, in the sense that $\Theta(\gamma) \subseteq \Theta_S$.
\end{theorem}
In particular, if $\zeta$ is an exposed point (i.e., if $S = \{\zeta\}$), then we can produce a directed geodesic. The statement of the theorem should be interpreted carefully: the argument we give here does not give a construction of such a $\gamma$ that is measurable with respect to the edge weights. Rather, it shows that \[\prob(\text{there exists an infinite geodesic directed in $\Theta_S$}) = 1\] by extracting an appropriate $(t_e)$-measurable event having probability one from a larger probability space.

In our construction, we build a limiting Busemann function $B$ and corresponding limiting geodesics via a particularly chosen limiting procedure. In our analysis of the asymptotics of this Busemman function, we also make a technical improvement on the methods of the last section. Note that our contradiction there came from the observation $B_\gamma$ grows (at least remaining positive) along its base geodesic $\gamma$ and becomes negative at a linear rate along the $-e_1$ axis. In fact, we can sharpen these observations by considering the growth of $B$ in a global sense. The main goal here will be to provide a version of \eqref{eq:hoferg2} which gives a ``shape theorem'' for a particular Busemman function which characterizes completely the linear-order growth.

We note here that the presentation of Theorem \ref{thm:dh} and its proof are influenced by the versions appearing in \cite{ADH16}; in particular, unlike the original paper \cite{DH14}, we make clear the generalization to $d > 2$.

\subsection{Construction of geodesics, Busemann functions}
%{\bf For notational convenience, we will assume $z = e_1$. However, this is not essential!}

Our construction benefits greatly from considering point-to-set geodesics. For a subset $S \subseteq \bbR^d$ and $x \in \Z^d$, let $T(x, S) = \inf_{y \in S} T(x,y)$. Recall the extension of $T$ to points of $\bbR^d$ --- we set $T(x,y) = T(x,[y])$, where $[y]$ is the unique point of $\Z^d$ such that $y \in [y] + [0,1)^d$]. A point-to-set geodesic from $x$ to $S$ is a path $\gamma$ from $x$ to some $[y]$ with $y \in S$, such that $T(\gamma) = T(x,S)$.

For each real number $\alpha$, let $H_\alpha = \{y: \, y \cdot \rho = \alpha\}$. It is not hard to see following the proof of Lemma \ref{lem:unique} that geodesics to $H_\alpha$ exist and are unique; we let $\Gamma_\alpha(x)$ denote the geodesic from $x \in \Z^d$ to $H_\alpha$, and $B_\alpha(x,y) = T(x, H_\alpha) - T(y, H_\alpha)$. We would be very happy if each $\Gamma_\alpha(x)$ and $B_\alpha(x,y)$ had limits as $\alpha \rightarrow \infty$, as we could then work with the limiting objects. Because it is not clear how to show this, we will take a particular sort of subsequential limit instead. Particular desiderata which guide the choice of limiting procedure are that the limits have appropriate translation-invariance properties (so we can do versions of averaging tricks and apply ergodic theorems) and that they preserve the relationship between the limiting analogues of $\Gamma_\alpha$ and $B_\alpha$ (so we can say something about geodesics from Busemann asymptotics).

 We consider the edge weight configuration $\omega = (t_e)_e$ to live on the canonical probability space $\Omega_1 := [0,\, \infty)^{\cE^d}$. We will need to consider an enlarged version of this space. Let $\vec{\cE}^d$ denote the set of 
 {\bf directed} edges (i.e., an ordered pair $(x,y)$ is in $\vec{\cE}^d$ if $\{x,y\} \in \cE^d$).  We let $\widetilde \Omega := \Omega_1 \times \Omega_2 \times \Omega_3$, where $\Omega_2 = \bbR^{\Z^d}$ and $\Omega_3 = \{0, 1\}^{\vec{\cE}^d}\ .$ For each $\alpha$, we will push forward our original measure $\prob$ on $\Omega_1$ to a measure $\mu_\alpha$ on $\widetilde \Omega$.

For $(x,y) = \vec{e} \in \vec{\cE}^d$, let $e = \{x,\,y\}$ be the undirected version of $\vec{e}$ and define the random variable 
\[\eta_\alpha(\vec{e}) :=
\begin{cases}
  1 \quad & \text{if } T(x, H_\alpha) = T(y, H_\alpha) + t_{e} \\
  0 \quad & \text{otherwise.}
\end{cases}
\]
For each $\alpha$ , define the map $\Phi_\alpha: \, \Omega_1 \rightarrow \widetilde \Omega$ as follows: 
\[\Phi_\alpha( \omega) = (\omega, B_\alpha, \eta_\alpha)\ .\]
For each $\alpha$, we define the probability measure $\mu_\alpha$ (on $\widetilde \Omega$ with the Borel sigma-algebra) to be the push-forward of $\prob$ by the map $\Phi_\alpha$.

$\mu_\alpha$ keeps track of the joint distribution of the edge weights, Busemann functions, and edges in geodesics. The geodesics are kept track of by the edge variables $\eta_\alpha$; for instance, we have almost surely that $\eta_\alpha((x,y)) = 1$ if and only if $(x,y)$ is in $\Gamma_\alpha(x)$. Indeed, if $(x,y) \in \Gamma_\alpha(x)$, then $\eta_\alpha((x,y)) = 1$ by definition of a geodesic; conversely, if $\eta_\alpha((x,y)) = 1$, then by concatenating $(x,y)$ with $\Gamma_\alpha(y)$ we produce a path $\gamma$ satisfying $T(\gamma) = T(x, H_\alpha)$ (which must be the unique geodesic).

Recall the translation operators from Definition \ref{defin:trans}. We extend them to $\widetilde \Omega$ as follows. For $z \in \Z^d$, consider a typical point $((t_e), (B(x,y)), (\eta_\alpha(\vec e)))$ of $\widetilde \Omega$; then
\[\theta_z((t_e), (B(x,y)), (\eta(\vec e))) = ((t_{e-z}), (B(x-z,y-z)), (\eta(\vec e - z))) \ .\]
\begin{lemma}\label{lem:mutrans}
  If $z \in \Z^d$ and $A \subseteq \widetilde \Omega$ is an event, then $\mu_{\alpha} \circ \theta_z^{-1} = \mu_{\alpha - \rho \cdot z}$.
\end{lemma}
\begin{proof}
  Let us just demonstrate via the event $\{\eta((x,y)) = 1\}$. We have
\begin{align*}
&\mu_\alpha \circ \theta_z^{-1}(\eta((x,y)) = 1) \\
&\qquad= \mu_\alpha( \eta((x+z,y+z)) = 1)= \prob((x+z,y+z) \in \Gamma(x+z, H_\alpha))\\
&\qquad= \prob((x,y) \in \Gamma(x, H_{\alpha - z \cdot \rho})) =  \mu_{\alpha - \rho \cdot z}(\eta((x,y)) = 1)\ . \end{align*}
Here we have used the invariance of $\prob$ under shifts, and the fact that $H_\alpha + z = H_{\alpha + \rho \cdot z}$.
\end{proof}

The following properties of $B_\alpha$ are proved via very similar arguments to those of Lemma \ref{lem:buseprops}, so we omit their proofs.
\begin{lemma}\label{lem:busenew}
  The following hold $\prob$-a.s.  for each $x,y,z \in \Z^d$:
  \begin{enumerate}
    \item[{\rm(1)}] $|B_\alpha(x,y)| \leq T(x,y)$;
  \item[{\rm(2)}] $B_\alpha(x,z) = B_\alpha(x, y) + B_\alpha(y,z)$;
    \item[{\rm(3)}] $B_\alpha (y,x) = -B_\alpha(x,y)$;
      \item[{\rm(4)}] If $y \in \Gamma_\alpha(x)$, then $B_\alpha(x,y) = T(x,y)$.
  \end{enumerate}
In particular, if we replace $B_\alpha$ and $\eta_\alpha$ by typical points $B$ and $\eta$ of $\Omega_2$ and $\Omega_3$, the analogues of the above hold $\mu_\alpha$-a.s.
\end{lemma}

When we consider a limit of $\mu_\alpha$, since we do not know that $\Gamma_\alpha$'s converge, there is no clear way to reconstruct an infinite geodesic corresponding to ``the $\alpha = \infty$ version of $\Gamma_\alpha$''. We wish to use the $\eta$ variables above to read off an infinite geodesic from a configuration on $\widetilde \Omega$. We will be greatly helped here by the following graphical construction.

\begin{definition}\label{defin:graph}
  For each $\eta \in \Omega_3$, define a directed graph $\mathbb{G}$ with vertex set $\Z^d$ and directed edge set as follows: $(x,y)$ is an edge of $\mathbb{G}$ if and only if $\eta((x,y)) = 1$. When the configuration $\eta$ is understood, we write $x \rightarrow y$ if there is a directed path in $\mathbb{G}(\eta)$ from $x$ to $y$.
\end{definition}
\begin{lemma}\label{lem:dirprops}The following hold for $\mu_\alpha$-a.e. configuration of $\widetilde \Omega$.
  \begin{enumerate}
  \item[{\rm(1)}] For each $x$, there is a directed path (possibly equal to $(x)$) from $x$ to $H_\alpha$.
    \item[{\rm(2)}] If $\gamma$ is a directed path in $\mathbb{G}$, then $\gamma$ is a geodesic.
      \item[{\rm(3)}] If $x \rightarrow y$, then $B(x,y) = T(x,y)$.
  \end{enumerate}
\end{lemma}
\begin{proof}
  Item (1) is clear by considering the fact that the edges $\vec {e}$ of $\Gamma_\alpha(x)$ have $\eta_\alpha(\vec{e}) = 1.$ Item (3) follows from the definition of $B_\alpha$ and property (4) of Lemma \ref{lem:busenew}.

Consider an outcome such that geodesics exist and are unique. To prove property (2), let $\gamma$ start at $x$ and traverse the directed edges $\vec{e}_1\, \ldots, \vec{e}_n$ in order. Let $K \leq n$ be the maximal index such that the path $\gamma_K$ which traverses (in order) $\vec{e}_1\, \, \ldots,\, \vec{e}_K$ is a subpath of $\Gamma_\alpha(x)$. We will show $K = n$.

By the observation that $\eta_{\alpha}((y,z)) = 1$ if and only if $(y,z) \in \Gamma_\alpha(y),$ we see that $\vec{e}_1$ is the first vertex of $\Gamma_\alpha(x)$ and so $K \geq 1$. We now show that if $K < n$, then $\gamma_{K+1}$ is also a subpath of $\Gamma_\alpha(x)$. Note that there is some path $\gamma'$ which extends $\gamma_{K}$ to a geodesic from $x$ to $H_\alpha$. Letting $\vec{e}_K = (y,z)$, the subpath of $\gamma'$ from $z$ to $H_\alpha$ must be the unique geodesic from $z$ to $H_\alpha$. On the other hand, since $\vec{e}_{K+1}$ is in $\mathbb{G}$ (and is thus in $\Gamma_\alpha(Z)$, we have that the edge of $\gamma'$ following $\vec{e}_K$ must be $\vec{e}_{K+1}$.
\end{proof}

We average the $\mu_\alpha$'s to produce a new measure $\mu_n^*$ on $\widetilde \Omega$: for $n = 1, 2,\ldots$, set
\[\mu_n^* = \frac{1}{n} \int_{0}^n \mu_\alpha\  d\alpha\ . \]
There is a technical argument required for this definition: namely, we need to show that for each measurable $A$, the map $\alpha \mapsto \mu_\alpha(A)$ is measurable (so the integral above makes sense).  We refer the interested reader to Appendix~A of \cite{DH14}.
\begin{lemma}\label{lem:subseq}
  There is a subsequence $(n_k)$ and a measure $\mu$ on $\widetilde \Omega$ such that
  \[\lim_{k \rightarrow \infty} \mu_{n_k}^* = \mu \quad \text{(weakly).} \]
\end{lemma}
\begin{proof}
  The distribution $\prob$ on $\Omega_1$ is easily seen to be tight, in the usual sense that for each $\varepsilon > 0$, we can find a compact measurable $K_\varepsilon \subseteq \Omega_1$ such that $\prob(\Omega_1 \setminus K_\varepsilon) < \varepsilon$. By property (1) of Lemma \ref{lem:busenew}, we see that the sequence $(\mu_n^*)$ is also tight. Prokhorov's theorem now gives the existence of a subsequential weak limit.
\end{proof}

Choose some $\mu$ as in the statement of Lemma \ref{lem:subseq}. This will be the object we use to construct the geodesic of Theorem \ref{thm:dh}. The bulk of this construction will be done in the next subsection, by analyzing the graphs $\mathbb{G}$ and Busemann functions $B$ sampled from $\mu$. For now, we just give the main reason for using the averaging procedure which constructed $\mu$ (and not, say, choosing $\mu$ just as a limit of $\mu_\alpha$).

\begin{lemma}\label{lem:transinv}
  $\mu$ is translation-invariant: for any $z \in \Z^d$ and any event $A \subseteq \widetilde \Omega$,
  \[\mu\circ \theta_z^{-1} (A) = \mu^*(A)\ . \]
\end{lemma}
\begin{proof}[Proof sketch]
  Let $n$ be a positive integer. We can write (using Lemma \ref{lem:mutrans})
  \[ \mu^*_n \circ \theta^{-1}_z( A) = \frac{1}{n} \int_{-z \cdot \rho}^{n - z \cdot \rho} \mu_\alpha (A) d\alpha\ .\]
In particular,
\[\left|\mu_n^* \circ \theta^{-1}_z( A) - \mu_n^*(A) \right| \leq \frac{1}{n}\left| \int_{n - z \cdot \rho}^{n} \mu_\alpha (A)\, d\alpha + \int_{-z \cdot \rho}^{0} \mu_\alpha(A)\, d \alpha \right|\ , \]
which tends to zero in $n$ for any fixed $A$. The result now follows by approximating $\mu$ by $\mu^*_n$ for $n$ large.
\end{proof}

\subsection{Asymptotics for samples from $\mu^*$}
We will now study the asymptotics of a sample $B \in \Omega_2$ from the marginal of $\mu$. Our goal is to replicate the previous ``averaging'' arguments but in a strong sense. Our first step, as before, is to control the expectation.
\begin{theorem}\label{thm:buseasymp}
  For any $x, y \in \Z^d$, we have
  \[\bfE_\mu B(x,y) = \rho \cdot (y-x)\ . \]
\end{theorem}
We need the following lemma:
\begin{lemma}\label{lem:linetime}
  Almost surely and in $L^1$,
  \[\lim_{\alpha \rightarrow \infty} T(0, H_\alpha)/\alpha = 1\ . \]
\end{lemma}
This lemma follows from the shape theorem in a fairly straightforward manner, so we omit the proof.
\begin{proof}[Proof of Theorem \ref{thm:buseasymp}]
  Let $n$ be a positive integer. Write
  \begin{align}
    \bfE_{\mu_n^*} B(-x, 0)
    &= \frac{1}{n} \left[\int_{0}^n \bfE T(-x, H_\alpha)\,d \alpha - \int_{0}^n \bfE T(0, H_\alpha)\, d \alpha \right]\nonumber\\
   %\nonumber\text{(by translation invariance)} 
   &= \frac{1}{n} \left[\int_{0}^n \bfE T(0, H_{\alpha + x \cdot \rho})\,d \alpha - \int_{0}^n \bfE T(0, H_\alpha)\, d \alpha \right]\nonumber\\
  &= \frac{1}{n} \left[\int_{n}^{n + x \cdot \rho} \bfE T(0, H_{\alpha})\,d \alpha - \int_{0}^{x \cdot \rho} \bfE T(0, H_\alpha)\, d \alpha \right]. \label{eq:shiftalph}
  \end{align}
The second equality comes by translation invariance.

Taking limits in \eqref{eq:shiftalph}, the second term goes to zero with $n$. To deal with the first term, note that Lemma \ref{lem:linetime} implies that for each $\alpha \in [0, x \cdot \rho]$,
\[\lim_{n \rightarrow \infty}\frac{\bfE T(0, H_{\alpha + n})}{n} = \lim_{n \rightarrow \infty} \frac{\bfE T(0, H_{\alpha + n})}{n + \alpha} \frac{n + \alpha}{n} = 1\ . \]
Applying this in \eqref{eq:shiftalph} (with a dominated convergence theorem argument to take the limit under the integral) gives $\bfE_{\mu_n^*} B(-x, 0) \rightarrow \rho \cdot x$.

It remains only to show that $\bfE_{\mu_{n_k}^*} B(-x,0) \rightarrow \bfE_{\mu} B(-x,0)$ (since the expectation at arguments $x,y$ now follows by translation-invariance). For $R > 0$, defining the continuous truncation
\[B_R(-x, 0) = B(-x, 0)\mathbf{1}_{|B(-x,0)| \leq R} + R\, \mathrm{sign}(B(-x,0)) \mathbf{1}_{|B(-x,0) > R|}\ , \]
we have by continuity and boundedness that 
\[\bfE_{\mu_{n_k}^*}B_R(-x,0) \rightarrow \bfE_{\mu} B_R(-x,0)\ .\]

The claim now follows by taking $R \rightarrow \infty$. Indeed, we have that $B(-x,0)^2 \leq T(0,-x)^2$, which has a finite second moment (bounded uniformly in $n$ by some constant $C$). The Cauchy-Schwarz inequality implies
\begin{align*}
\bfE_{\mu_n^*} |B(-x,0)| \mathbf{1}_{|B(-x,0)| > R} &\leq  \left(\bfE_{\mu_n^*} |B(-x,0)|^2\right)^{1/2}  \mu_n^*(|B(-x,0)| > R)^{1/2}\\
&\leq C^{1/2} \prob(T(0,-x) > R)^{1/2}\ ,
\end{align*}
where we again use $|B(-x,0)| \leq T(-x,0)$. The probability above goes to zero uniformly in $n$ as $R \to \infty$, so we have established
\[\lim_{R \rightarrow \infty} \limsup_{n \rightarrow \infty}\bfE_{\mu_n^*} |B(-x,0)| \mathbf{1}_{|B(-x,0)| > R} = 0\ .\qedhere \]
\end{proof}

As we said earlier, we want to prove a 
\index{shape theorem}%
shape theorem for $B$, establishing the a.s. leading-order behavior of $B(0,x)$ under $\mu$. The preceding lemma suggests a target for this shape theorem: for $|x|$ large, one expects $|B(0,x) - \rho \cdot x| \ll |x|\ .$ Unfortunately, a result like this does not immediately follow in general without information about $\cB$, and it is here we use assumption (Dif).

A brief explanation of the problem is as follows. We want to use the ergodic theorem, as in the proof of Garet-Marchand's result, to establish that $B(0, nx)/n \rightarrow \rho \cdot x$ for each fixed $x$; if we could do this, we could patch together a global result by continuity. Unfortunately, the ergodic theorem does not give convergence of $B(0, nx)/n$ to a deterministic limit, because the measure $\mu$ is not guaranteed to be ergodic. So the best we can hope for {\it a priori} is convergence to some functional whose mean is $\rho$. 

The above obstacle is the main reason we assumed (Dif) in the first place. We will see that the differentiability gives us the ability to say the random limiting functional is in fact identical to $\rho$, almost surely.

\begin{theorem}\label{thm:dhshape}
  There exists a random vector $\varpi$ on $\widetilde \Omega$  such that
  \[\mu\left(\limsup_{|z| \rightarrow \infty} \frac{|B(0,z) - \varpi \cdot z|}{|z|} = 0  \right) = 1\ . \]
  Moreover, $\bfE_{\mu} \varpi = \rho$.
\end{theorem}
We give a sketch of the main ideas of the proof, highlighting the issues with ergodicity.
\begin{proof}[Proof sketch]
  We first show that for $e = e_1, e_2, \ldots, e_d$, we have
  \begin{equation}
    \label{eq:axislim}
    \lim_{n \rightarrow \infty} B(0, n e)/n =: \varpi(e) \quad \text{exists $\mu$-a.s.}
  \end{equation}
To establish \eqref{eq:axislim}, we write $n^{-1}B(0, n e) = n^{-1}\sum_{j=1}^n B((j-1)e, je)$, which follows from (2) in Lemma \ref{lem:busenew} after some work to pass this property through the limit which produces $\mu$. We then rewrite this sum as a sum of copies of $B(0, e)$ evaluated in shifted environments as usual, and apply the ergodic theorem. Since $\mu$ is not ergodic but rather merely translation-invariant, we are only guaranteed that the limit exists and defines some random variable: it need not be constant.

We now define $\varpi = (\varpi(e_1), \ldots, \varpi(e_d))$; $\varpi$ is invariant under translations of the realization $((t_e), (B), (\eta))$. Our next step is to see that for each fixed $x = (x(1), x(2), \ldots, x(d)) \in \Z^d$, we have $B(0, nx)/n \rightarrow x \cdot \varpi$. This follows by writing 
\[B(0, nx) = B(0, n x(1) e_1) + B(n x(1) e_1, n x(1) e_1 + n x(2) e_2) + \ldots\]
and using the previous convergence result (and invariance of $\varpi$) to approximate each of these terms by $n x(1) \varpi(1)$, $n x(2) \varpi(2)$, etc.

This gives us convergence in fixed directions. To give the global convergence as in the statement of the theorem, we follow the proof of the shape theorem (just as we did in the proof of \eqref{eq:hoferg2}). The form of the mean of $\varpi$ is a consequence of Theorem \ref{thm:buseasymp}.
\end{proof}

As promised, we claimed that we can in fact show that $\varpi$ is $\rho$ under our assumptions. We conclude this subsection by giving the argument.
\begin{lemma}\label{lem:gottadif}
  $\mu$-a.s., the hyperplane $H_{\varpi}:= \{x: x \cdot \varpi = 1\}$ is a supporting hyperplane for $\partial \cB$ at $\zeta$. In particular, under (Dif), we have $\varpi = \rho$ almost surely.
\end{lemma}
\begin{proof}
  Note that almost surely, for any fixed $x \in \cB$
\[\varpi \cdot x = \lim_{n \rightarrow \infty}\frac{B(0, nx)}{n} \leq \lim_{n \rightarrow \infty} \frac{T(0, nx)}{n} =  g(x) \leq 1\ . \]
In particular, $\cB$ a.s. lies on one side of $H_{\varpi}$. On the other hand,
\[\bfE_{\mu} \varpi \cdot \zeta = \rho \cdot \zeta = 1\ . \]
Thus $\varpi \cdot \zeta = 1$ almost surely, and $H_\varpi$ is a supporting hyperplane.
\end{proof}

\subsection{Directedness}
We now prove Theorem \ref{thm:dh}. We show that for $\mu$-almost every element of $\widetilde \Omega$, there is an infinite path from $0$ in $\mathbb{G}(\eta)$ (which must be a geodesic, by the limiting version of Lemma \ref{lem:dirprops} (2)) and which has the required directedness.

\begin{proof}[Proof of Theorem \ref{thm:dh}]
It is not hard to show a limiting version of Property (1) of Lemma \ref{lem:dirprops} which says that with $\mu$-probability one, there is an infinite path in $\mathbb{G}$ from $0$. As mentioned just prior, this path is also easily seen to be a geodesic. So the main argument is to show
\begin{equation}
  \label{eq:lastdir}
  \text{a.s., for each infinite path $\gamma$ in $\mathbb{G}$, } \Theta(\gamma) \subseteq \Theta_S\ .
\end{equation}

Let $\gamma = (0 = x_0, x_1, \ldots)$, and let $(x_{n_k})$ be a subsequence such that $x_{n_k}/|x_{n_k}| \rightarrow \theta$. We show $\theta \in \Theta_S$. Applying Theorem \ref{thm:dhshape} (with $\varpi = \rho$) gives
\[ \lim_{k} B(0, x_k) / |x_k| = \rho \cdot \theta\ . \]
On the other hand, this limit also equals $\lim_k T(0, x_k) / |x_k| = g(\theta)$. In particular, $\theta / g(\theta)$ is on $\partial \cB$ and in the set $H$, so it is in $S$.
\end{proof}

\subsection{Sector-directed geodesics and uniqueness\label{sec:lnredux}}
We return in this section to the case of $\Z^2$. It is first worth discussing the result of Theorem \ref{thm:dh} and reflecting on the directedness conjectures of the model. Since it is widely believed that the boundary of the limit shape should be uniformly curved and everywhere differentiable (at least for typical edge weight distributions), one expects that there should exist directed geodesics in each direction. On the other hand, Theorem \ref{thm:dh} (as well as Theorem \ref{thm:newman}) would require strong and currently unknown information about $\cB$ in order to prove the existence of even a single geodesic having asymptotic direction $e_1$.

In fact, Theorem \ref{thm:dh} also applies to the setting where the joint distribution of the 
\index{edge weights}%
$t_e$'s is not i.i.d., but rather only assumed translation-ergodic. In this generality, one cannot expect all of the usual FPP conjectures to hold: for instance, we know any ``reasonable'' set is attainable as a limit shape (see Theorem 3.8 in \cite{ch:Damron}). %HERE
In the case of an ergodic weight distribution having a polygonal $\cB$, Theorem \ref{thm:dh} guarantees only the existence of some $\cB$-dependent finite number of geodesics and does not rule out their wandering across entire sectors. In recent work by Brito and Hoffman \cite{BH17}, an ergodic weight distribution exhibiting this kind of counterintuitive behavior is constructed: in their example, there exist exactly four infinite geodesics, and each wanders across an entire quadrant of the lattice. See also the recent work by Alexander and Berger \cite{AB17}, where an ergodic distribution is constructed having an octogon as its limit shape, but having geodesics directed only toward the axis directions of the lattice.

With these ideas in mind, we return to the problem raised in Section \ref{sec:lnuniq}. Recall that Licea and Newman showed that for ``typical'' directions $\theta$ --- i.e., all $\theta$ in some set $D \subseteq [0, 2 \pi)$ --- there is a.s. a unique $\theta$-directed geodesic. We will describe work from \cite{DH17} which extends these uniqueness statements using Busemann function techniques. In keeping with the above discussion, these results give not uniqueness of geodesics having some single direction $\theta$, but rather uniqueness of geodesics directed within sectors (as produced by Theorem \ref{thm:dh}). The result also requires further assumptions on the limit shape; we comment on these after the statement of the theorem.
\begin{theorem}[\cite{DH17}]\label{thm:bigeo} Consider FPP on $\Z^2$ under the standard assumptions of Assumption \ref{asn:std}. Assume that $\partial \cB$ is differentiable everywhere. Let $\zeta \in \partial \cB$; let $H$ be the line tangent to $\cB$ at $\zeta$, let $S$ be the sector of of contact between $H$ and $\cB$, and let $\Theta_S$ be the corresponding set of angles (see \eqref{eq:difang} above). Then there is a.s. a unique geodesic $\gamma_x$ from each $x$ such that $\gamma_x$ is directed in $\Theta_S$ (i.e., $\Theta(\gamma_x) \subseteq \Theta_S$). Moreover, these geodesics a.s. 
\index{geodesic!coalescence}%
\index{geodesic!merging}%
coalesce: $|\gamma_x \triangle \gamma_y| < \infty$ for all $x,y$.
\end{theorem}

In fact, the theorem requires only a weaker assumption of differentiability at $\zeta$ and the endpoints of $S$. Under the additional assumption that every point of $\partial \cB$ is also an exposed point (so that $S = \{\zeta\}$ in the theorem), the statement gives us an existence and uniqueness statements of geodesics that are directed in the sense of Theorem \ref{thm:newman}. We describe here the outline of the proof of Theorem \ref{thm:bigeo} with many of the more graph-theoretic arguments truncated, preferring to emphasize the role of Busemann functions in the analysis.
\begin{proof}[Sketch of proof of Theorem \ref{thm:bigeo}]
One of the hardest parts of the argument is dealing with the undirectedness of the model. This will be dealt with by considering both FPP on $\Z^2$ and on the right half-plane $\mathbb{H} = \{(a,b) \in \Z^2: \, a \geq 0\}$; the edges of $\mathbb{H}$ are the edges of $\Z^2$ having both endpoints in $\mathbb{H}$. The models on the two graphs can be coupled in the natural way by using the same edge weights. The main reason for introducing $\mathbb{H}$ is that two geodesics cannot change their relative top / bottom ordering without crossing each other, and geodesic intersections yield a useful monotonicity property for the corresponding Busemann functions.

For simplicity, we consider just the sector in the $e_1$-direction (i.e., $\zeta$ is taken to be a multiple of $e_1$); as before, this allows us to simplify notation without sacrificing any important parts of the argument. Note that our assumption implies $\Theta_S \subseteq (-\pi/4, \pi/4)$. Let $G_H(x,y)$ denote the geodesic between $x,y \in \mathbb{H}$ for the FPP model on $\mathbb{H}$.\medskip

{\bf Step 1: Finding ``trapping'' geodesics.} We first establish the following two statements:
\begin{enumerate}
\item \label{item:one} There exist sequences $(\Theta_B^{n})_{n=1}^\infty$ and $(\Theta_T^{n})_{n=1}^\infty$ of sets of angles in the interval $(-\pi/2, \pi/2)$ such that $\sup \Theta_T^{n+1} < \inf \Theta_T^{n}$ and $\inf \Theta_B^{n+1} > \sup \Theta_B^n$ such that $(\inf \Theta_B^n)$ converges to the bottom (i.e. clockwise) endpoint of $\Theta_S$ and $(\sup \Theta_T^n)$ converges to the top (counterclockwise) endpoint of $\Theta_S$, such that the following holds. For each $n$, there a.s. exist geodesics in $\cT(0)$ directed in $\Theta_T^n$ and $\Theta_B^n$. 
\item The statement of Item 1 holds on $\mathbb{H}$, where the geodesics mentioned there are geodesics in $\mathbb{H}$ for the FPP model defined on $\mathbb{H}$.
\end{enumerate}
The first item above follows immediately from Theorem \ref{thm:dh}, using the fact that $\partial \cB$ is differentiable at the endpoints of $S$. 

To establish the result on the half-plane, we use the result on $\Z^2$. The geodesic in $\cT(0)$ from item \eqref{item:one} directed in $\Theta_T^n$ eventually crosses each vertical line; from this we derive that there is  a.s. a positive density of sites along the $e_2$-axis having $\mathbb{H}$-geodesics directed in $\Theta_T^n$. A similar statement holds for $\Theta_T^{n+1}$. Let $k e_2$ and $- \ell e_2$ ($k, \ell > 0$) be two vertices of the $e_2$-axis having such $\Theta_T^n$ and $\Theta_T^{n+1}$-directed (respectively) $\mathbb{H}$-geodesics. Denote these by $\Gamma^n$ and $\Gamma^{n+1}$; the directedness of these geodesics moreover allows us to choose them so that $\Gamma^n \cap \Gamma^{n+1} = \varnothing$.

Now choose some sequence $(v_n) \subseteq \mathbb{H}$ such that $v_n / |v_n| \to \theta$ for some $\theta \in (\sup \Theta_T^{n+1}, \inf \Theta_T^n)$. Let $\Gamma$ be some subsequential limit of the sequence $(G_H(0, v_n))$.  Uniqueness of passage times forces $\Gamma$ to be contained in the region containing $0$ and delimited by $\Gamma^n$, $\Gamma^{n+1}$, and the segment of the $e_2$-axis between $ - \ell e_2$ and $k e_2$. Thus $\Gamma$ is directed in $[\sup \Theta_T^{n+1}, \inf \Theta_T^n]$, and continuing in this way we can find an ordered sequence of sectors converging to the top endpoint of $\Theta_S$ and half-plane geodesics from $0$ directed in these sectors; a similar argument establishes the analogous statement for bottom endpoint sectors.\medskip

{\bf Step 2: Existence of extremal geodesics.} We define subtrees of $\cT(0)$ and its half-plane analogue consisting of $S$-directed geodesics, then find topmost and bottommost geodesics in them. More generally, for any $x$, let $\mathcal{G}(x)$ and $\mathcal{G}_H(x)$ be the sets of $\Theta_S$-directed $\Z^2$- and $\mathbb{H}$-geodesics starting at $x$, respectively. We define an ordering $\prec$ on each of these sets as follows: $\Gamma \prec \Gamma'$ if the topmost intersection of $\Gamma$ with $\{y: y \cdot e_1 = n\}$ is below the corresponding topmost intersection of $\Gamma'$, for all large $n$. This $\prec$ turns out to a.s. define a total ordering, due to the directedness of these geodesics (so they eventually pass each vertical line) and the uniqueness of passage times (which prevents crossing and then re-crossing).

Given this ordering, one shows that there a.s. exist a unique maximal element $\Gamma_T(x) \in \mathcal{G}(x)$ and minimal element $\Gamma_B(x) \in \mathcal{G}(x)$, with analogous statements for $\mathbb{H}$ (we denote the $\mathbb{H}$ geodesics by $\Gamma_T^H$ and $\Gamma_B^H$). To do this, one constructs $\Gamma_T(x)$ explicitly as the limit of $G(x, x_T(n))$, where $x_T$ is the element of $\mathcal{G}(x) \cap \{y: y \cdot e_1 = n\}$ having largest $e_2$-coordinate. One argues that any such subsequential limit must be $S$-directed (being trapped by the geodesics from Item 1 via uniqueness of passage times) but asymptotically lie above any other $S$-directed geodesic by virtue of the topmost position of $x_T(n)$ and the fact that geodesics in $\mathcal{G}(x)$ cannot backtrack too far (to interchange orderings). \medskip

{\bf Step 3: Coalescence of extremal geodesics.} One argues that the family $\{\Gamma_T(x)\}_x$ is a coalescing family of geodesics, with analogous statements for the $\Gamma_B(x)$'s and their half-plane analogues. The $\Z^2$ coalescence statements follow from a version of the Licea-Newman argument (see Section \ref{sec:lnuniq} above). The limiting procedure used to produce topmost geodesics shows that a subsegment of a topmost geodesic is also a topmost geodesic from its initial vertex. Since there is a unique $\Gamma_T$ from each $x$, if there were multiple distinct $\Gamma_T$ geodesics, one could use them and an edge-modification argument to construct geodesic structures which would be forbidden by translation-invariance. 

The coalescence result on $\mathbb{H}$ follows from the coalescence result on $\Z^2$; we describe the case that the geodesics start on the $e_2$-axis. Since the $\Z^2$ geodesics $\Gamma^T(x)$ pass every vertical line, we can find a positive density of sites along the $e_2$-axis whose topmost $\Z^2$ geodesics stay in $\mathbb{H}$. In such a case, it turns out that $\Gamma^T(x)$ is also the topmost $S$-directed geodesic from $x$ in $\mathbb{H}$. Now the fact that these geodesics coalesce also forces every other $\Gamma_T^H$ geodesic from axial vertices to coalesce with this family. The key observation forcing the trapping is again that subsegments of topmost geodesics are topmost geodesics, so the aforementioned topmost geodesics cannot cross.

Given these coalescing families, we can define Busemann functions $B_T$ and $B_B$ and their half-plane analogues $B_T^H$ and $B_B^H$. If $x \in \Z^2$ is any vertex and $\Gamma_T(x) = (x = x_1, x_2, \ldots)$, we set
\[B_T (y,z) = \lim_{n \to \infty} [T(y, x_n) - T(z, x_n) ]\ . \]
By coalescence, similarly to item (5) of Lemma \ref{lem:buseprops}, the definition above does not depend on the choice of $x$. We make similar definitions for the other Busemann functions. \medskip

{\bf Step 4: Limiting behavior of Busemann functions.}
One follows an argument similar to that used in the proof of Theorem \ref{thm:dhshape} to establish 
\index{shape theorem}%
shape theorems for $B_T$ and $B_B$. These say that there exists a (deterministic, since $B_T$ depends only on the i.i.d. environment) vector $\rho_T$ such that, for each $\delta > 0$,
\[ \prob\left(\limsup_{|x| \to \infty} \left|\frac{B_T(0,x) - \rho_T \cdot x}{|x|} \right| > \delta \right) = 0\ .  \]
Again a similar statement holds with $B_B$ and some $\rho_B$. Just as in the preceding Lemma \ref{lem:gottadif}, we see that $\{y: \rho_T \cdot y = 1\}$ must be the unique tangent line to the limit shape at $\zeta$ and similarly for $\rho_B$, giving $\rho_T = \rho_B$.

Since the tangent line at $\zeta$ is vertical, $\rho \cdot e_2 = 0$. Note that $B_T^H(0, k e_2) /k \to \bfE B_T^H(0, e_2)$ by the ergodic theorem. One can now show that $B_T^H(0, k e_2) = o(k)$ using the fact that $\Gamma_T(\ell e_2) = \Gamma_T^H(\ell e_2)$ for infinitely many $\ell$ and then applying the above shape theorem. An analogous statement holds for $B_B^H$, so defining $\Delta(x,y) = B_T^H(x,y) - B_B^H(x,y)$, we have $\bfE \Delta(0, e_2) = 0$.\medskip

{\bf Step 5: Deriving a contradiction.} We assume that $\Gamma_B(0) \neq \Gamma_T(0)$ with positive probability to derive a contradiction. We show that $\Delta(0,e_2)$ has a definite sign, which combined with the previous step gives $\Delta(0, e_2) = 0$ a.s. On the other hand, this is impossible if  $\Gamma_B(0)$ and $\Gamma_T(0)$ are not identical almost surely.

To see the ``definite sign'' claim, we note that the coalescence of top and bottom geodesics, along with the fact that geodesics from axis points in $\mathbb{H}$ must cross to change their top-bottom ordering, gives that $\Gamma_T^H(0)$ and $\Gamma_B^H(e_2)$ must share a vertex $z$. We will use the vertex $z$ to ``switch paths'' and bound the Busemann functions via subadditivity. Let $a_T$ and $a_B$ be vertices along $\Gamma_T^H(0)$ and $\Gamma_B^H(0)$, respectively, which come after both $z$ and the coalescence points with $\Gamma_T^H(e_2)$ and $\Gamma_B^H(e_2)$. We note that 
\begin{align}
B_T(0, e_2) = T(0, a_T) - T(e_2, a_T); \quad B_B(0, e_2) = T(0, a_B) - T(e_2, a_B)\ . \label{eq:breakpaths}
\end{align}

We now compute
\begin{align*}
  T(0, a_B) + T(e_2, a_T) &\leq T(0, z) + T(z, a_B) + T(e_2, z) + T(z, a_T)\\
  &= T(0, a_T) + T(e_2, a_B)\ .
\end{align*}
The inequality above follows from subadditivity of $T$; the equality comes from the geodesics $\Gamma_T^H(0)$ and $\Gamma_B^H(e_2)$ (so that, for instance, $T(0, z) + T(z, a_T) = T(0, a_T)$). Subtracting the terms involving $e_2$ from both sides of the above and using \eqref{eq:breakpaths} gives
\begin{equation}
  \label{eq:busemono}
  B_B(0, e_2) \leq B_T(0, e_2)\ , \quad \text{so} \quad \Delta(0, e_2) \geq 0 \quad \text{ a.s.}
\end{equation}
Combining \eqref{eq:busemono} with the fact that $\bfE \Delta(0, e_2) = 0$ gives that 
\begin{equation} \label{eq:deltazero}
\Delta(0, e_2) = 0\quad \text{a.s.}
\end{equation}

We contradict \eqref{eq:deltazero} with the following estimate:
\begin{lemma}
  \label{lem:deltaunderdistinct}
  If $\prob(\Gamma_B(0) \neq \Gamma_T(0)) > 0$, then $\prob(\Delta(0, e_2) > 0) > 0$.
\end{lemma}
Note that together, Lemma \ref{lem:deltaunderdistinct} and \eqref{eq:deltazero} show that $\prob(\Gamma_B(0) = \Gamma_T(0)) = 1$. Since every unigeodesic from $0$ directed in $\Theta_S$ asymptotically lies between $\Gamma_B(0)$ and $\Gamma_T(0)$ under the order $\prec$, we see that $\Gamma_B(0) = \Gamma_T(0)$ a.s. implies that there is a.s. at most one $\Theta_S$-directed geodesic from $0$.

\begin{proof}[Proof of Lemma \ref{lem:deltaunderdistinct}]
  Suppose for the sake of contradiction that (under the assumption of the lemma) $\prob(\Delta(0, e_2) = 0) = 1$. Since $B_T$ and $B_B$ are additive (compare to the second property of Lemma \ref{lem:buseprops}), we have
  \begin{equation}
    \label{eq:deltalarge}
    \prob(\Delta(0, k e_2) = 0) = 1 \quad \text{for all integers } k \geq 1\ .
  \end{equation}
We show that \eqref{eq:deltalarge} contradicts the a.s. uniqueness of passage times.

Choosing $b_1$ sufficiently large, we can find a $\delta >0$ and an edge $f \in [0, b_1] \times [-b_1, b_1]$ such that with probability at least $\delta$, $f \in \Gamma_T^H(0) \triangle \Gamma_B^H(0)$ (we assume that $f \in \Gamma_T^H(0)$, the other case being similar). By directedness of the half-plane geodesics in $\Theta_S$, we can choose a $k > b_1$ such that
\[\prob\left(\left(\Gamma_T^H(k e_2) \cup \Gamma_B^H(k e_2) \right) \cap \left([0, b_1] \times [-b_1, b_1]\right) = \varnothing \right) > 1- \delta / 2\ . \]
In particular, on the above event, $f$ is not in $\Gamma_T^H(k e_2)$ or $\Gamma_B^H(k e_2)$.

We have shown that with probability at least $\delta / 2,$ the edge $f$ is in $\Gamma_T^H(0)$ but not $\Gamma_B^H(0)$, $\Gamma_T^H(k e_2)$, or $\Gamma_B^H(k e_2)$, and furthermore $\Delta (0, k e_2) = 0$. Recalling the coalescence properties of these half-plane geodesics, let $a_T \in \Gamma_T^H(0) \cap \Gamma_T^H(k e_2)$ and similarly for $a_B$; we have
\begin{align*}
  \Delta(0, k e_2) &= B_T(0, k e_2) - B_B(0, k e_2)\\
  &= T(0, a_T) - T(k e_2, a_T) - T(0, a_B) + T(k e_2, a_B) = 0\ .
\end{align*}
Only one of the terms above involves $t_f$. This contradicts the almost sure uniqueness of passage times from Lemma \ref{lem:unique}.
\end{proof}
Theorem \ref{thm:bigeo} is proved.
\end{proof}

\bibliographystyle{amsplain}

%\bibliographystyle{plain}
%\bibliography{busemann}

\end{document}